\documentclass[10pt,reqno]{amsart}
\usepackage{amsmath}
\usepackage{amsfonts}
\usepackage{amssymb}
\usepackage{amsthm}
\usepackage{palatino}
\usepackage[margin=2.5cm, vmargin={1.5cm},includefoot]{geometry}

\usepackage{pst-plot}
\usepackage{pst-grad}
\usepackage{pstcol}
\usepackage{color}

\newrgbcolor{LemonChiffon}{1.0 0.98 0.8}
\newrgbcolor{magenta}{1.0 0.3 0.4}
\newrgbcolor{mistyrose}{1.0 0.0 0.6}
\newrgbcolor{deeppink}{1.0 0.08 0.58}
\newrgbcolor{lightsalmon}{1.0 0.63 0.48}
\newrgbcolor{lightred}{0.7 0.0 0.0}
\newrgbcolor{lightblue}{0.68 0.85 0.90}
\newrgbcolor{indianred}{0.80  0.36  0.36}
\newrgbcolor{lightgreen}{0.56 0.93 0.56}
\newrgbcolor{byellow}{0.93 0.76 0.3}

\begin{document}

\title[Decomposition of the restricted partition generating function]{On the partial fraction decomposition of the restricted partition generating function}
\author{Cormac O'Sullivan}
\address{Department of Mathematics\\
 The CUNY Graduate Center \\ New York, NY 10016-4309\\ U.S.A.}
\email{cosullivan@gc.cuny.edu}
\date{Sept 1, 2012}

\thanks{Support for this project was provided by a PSC-CUNY Award, jointly funded by The Professional Staff Congress and The City University of New York.}

\begin{abstract}
We  provide new formulas for the coefficients in the partial fraction decomposition of the restricted partition generating function. These techniques allow us to partially resolve a recent conjecture of Sills and Zeilberger. We also describe upcoming work, giving a resolution to Rademacher's conjecture on the asymptotics of these coefficients.
\end{abstract}

\maketitle

\def\s#1#2{\langle \,#1 , #2 \,\rangle}

\def\H{{\mathbf{H}}}
\def\F{{\frak F}}
\def\C{{\mathbb C}}
\def\R{{\mathbb R}}
\def\Z{{\mathbb Z}}
\def\Q{{\mathbb Q}}
\def\N{{\mathbb N}}
\def\G{{\Gamma}}
\def\GH{{\G \backslash \H}}
\def\g{{\gamma}}
\def\L{{\Lambda}}
\def\ee{{\varepsilon}}
\def\K{{\mathcal K}}
\def\Re{\mathrm{Re}}
\def\Im{\mathrm{Im}}
\def\PSL{\mathrm{PSL}}
\def\SL{\mathrm{SL}}
\def\Vol{\operatorname{Vol}}
\def\lqs{\leqslant}
\def\gqs{\geqslant}
\def\sgn{\operatorname{sgn}}
\def\res{\operatornamewithlimits{Res}}
\def\li{\operatorname{Li_2}}
\def\lis{\operatorname{Li_2^*}}
\def\clp{\operatorname{Cl}'_2}
\def\clpp{\operatorname{Cl}''_2}

\newcommand{\stira}[2]{{\genfrac{[}{]}{0pt}{}{#1}{#2}}}
\newcommand{\stirb}[2]{{\genfrac{\{}{\}}{0pt}{}{#1}{#2}}}
\newcommand{\norm}[1]{\left\lVert #1 \right\rVert}
\newcommand{\sect}[1]{{\Large \section{\bf #1}}}

\newtheorem{theorem}{Theorem}[section]
\newtheorem{lemma}[theorem]{Lemma}
\newtheorem{prop}[theorem]{Proposition}
\newtheorem{conj}[theorem]{Conjecture}
\newtheorem{cor}[theorem]{Corollary}

\newcounter{coundef}
\newtheorem{adef}[coundef]{Definition}

\renewcommand{\labelenumi}{(\roman{enumi})}

\numberwithin{equation}{section}

\bibliographystyle{plain}

\section{Introduction}
Let $p_N(n)$ denote the number of  partitions of $n$ into at most $N$ parts. As Euler showed, these restricted partitions have the generating function
\begin{equation}\label{eul}
\sum_{n=0}^\infty p_N(n) q^n = \prod_{j=1}^N \frac{1}{1-q^j}.
\end{equation}
Rademacher's $C_{hk\ell}(N)$ coefficients, see \cite[Eq. (130.5)]{Ra}, are uniquely defined  in the partial fraction decomposition of this generating function:
\begin{equation}\label{tp}
\prod_{j=1}^N \frac{1}{1-q^j}=\sum_{\substack{0\leqslant h<k \leqslant N \\ (h,k)=1}}
\sum_{\ell=1}^{\lfloor N/k \rfloor} \frac{C_{hk\ell}(N)}{(q-e^{2\pi ih/k})^\ell}.
\end{equation}
For a given $N$, knowing all the coefficients $C_{hk\ell}(N)$ allows us to express $p_N(n)$ as a finite sum,
\begin{equation}\label{eul2}
p_N(n)= \sum_{\substack{0\leqslant h<k \leqslant N \\ (h,k)=1}}
\sum_{\ell=1}^{\lfloor N/k \rfloor} C_{hk\ell}(N) \binom{\ell-1+n}{\ell-1} (-1)^\ell e^{-2\pi i h(\ell+n)/k},
\end{equation}
 where \eqref{eul2} follows from using the binomial theorem to write the right side of \eqref{tp} as a power series in $q$, and equating coefficients with the left side of \eqref{eul}.

So, for example, when $N=2$ we have
\begin{equation*}
    \frac 1{(1-q)(1-q^2)} = \frac{C_{011}(2)}{q-1}+\frac{C_{012}(2)}{(q-1)^2}+\frac{C_{121}(2)}{q+1}
\end{equation*}
for $C_{011}(2)=-1/4$, $C_{012}(2)=1/2$, $C_{121}(2)=1/4$ and \eqref{eul2} implies $p_2(n) = (2n+3+(-1)^n)/4$. As in \cite[pp. 221-222]{Ra}, this may be written as $p_2(n) = \lfloor n/2 \rfloor +1$.

Writing in \cite[p. 301]{Ra}, Rademacher lamented the lack of formulas for $C_{hk\ell}(N)$.  Andrews, in \cite{An}, provided the first one as we see later in \eqref{Andrews}. However, Andrews' formula does not allow easy calculation.  Very recently, Sills and Zeilberger in \cite{SZ} showed a fast recursive method to compute $C_{hk\ell}(N)$ for large $N$ and, when the difference $N-\ell$ is fixed, they  solved the recursion to prove formulas like
\begin{equation*}
    C_{01N}(N)= \frac{(-1)^N}{N!}, \qquad C_{01(N-1)}(N)= \frac{(-1)^{N+1}}{4(N-2)!},
\end{equation*}
as we describe in Section \ref{zel}.

In this paper we provide many new and  explicit expressions for the coefficients $C_{hk\ell}(N)$. Section \ref{tw} develops formulas for the simplest case $h/k=0/1$, showing directly that $C_{01\ell}(N)$ is always rational. For example,
\begin{equation}\label{c011}
    C_{011}(N) = \frac{(-1)^N }{ N!} \sum_{j_0+j_1+j_2+ \cdots + j_N = N-1}
     B_{j_1}B_{j_2}  \cdots  B_{j_N} \frac{ 1^{j_1} 2^{j_2} \cdots  N^{j_N}}{ j_0! j_1 !
j_2 ! \cdots j_N!}
\end{equation}
is a special case of \eqref{c01ra}, with $B_j$ the $j$th Bernoulli number. In Sections \ref{next} - \ref{s4} we treat the general case, describing the closely related work of Sylvester and Glaisher in Section \ref{sylvv}. The intriguing conjecture of Rademacher on the behavior of  $C_{hk\ell}(N)$ as $N \to \infty$ is discussed in Section \ref{rcoj}. This old conjecture has motivated much of the study of these coefficients and we describe a forthcoming result on the asymptotics of $C_{011}(N)$ that should in fact disprove it. In Section \ref{zel} we use the techniques we have developed to partially resolve a conjecture of Sills and Zeilberger, and in the last section  Andrews' method is extended to find further formulas for $C_{hk\ell}(N)$.

\section{Initial Formulas for Rademacher's coefficients} \label{tw}
From the definition \eqref{tp} we see
$$
C_{hk\ell}(N) = \res_{q=e^{2\pi i h/k}} \left(q - e^{2\pi i h/k}\right)^{\ell-1} \prod_{j=1}^N \frac{1}{1-q^j}.
$$
With the natural change of variables $q=e^{2\pi i z}$ we obtain
\begin{equation}\label{key}
    C_{hk\ell}(N) = 2\pi i \res_{z = h/k} \frac{e^{2\pi i z}\left(e^{2\pi i z} - e^{2\pi i h/k}\right)^{\ell-1}}{(1-e^{2\pi i z})(1-e^{2\pi i 2z}) \cdots (1-e^{2\pi i N z}) }
\end{equation}
where we used that
\begin{equation*}
    \res_{q=e^{2\pi i c}}f(q) = \res_{z = c}f(e^{2\pi i z}) \cdot 2 \pi i \cdot e^{2\pi i z}
\end{equation*}
which is implied by part (i) of the following result of Jacobi on residue composition (with $g(z)=e^{2\pi i z}$).

\begin{theorem} \label{rc}
Suppose $g(z)$ is  holomorphic in a neighborhood of $z=c$ and suppose $f(z)$ is meromorphic in a neighborhood of $z=g(c)$.
\begin{enumerate}
\item If $g'(c) \neq 0$ then
$$
\res_{z = g(c)}f(z) = \res_{z = c}f(g(z)) g'(z).
$$
\item More generally, if $g(z)-g(c)$ has a zero of order $m$ at $z=c$ then
$$
m \res_{z = g(c)}f(z) = \res_{z = c}f(g(z)) g'(z).
$$
\end{enumerate}
\end{theorem}

For the proof (it is really a result in formal power series) see \cite[Theorem 1.2.2, p. 15]{GJ} or the original \cite{Ja}.
The above derivation of \eqref{key} is based on an almost identical calculation  due to Beck, Gessel and Komatsu in \cite[pp. 3-4]{Be1} where they
derive formulas for  the `polynomial part' of the restricted partition function.  In fact, as a referee has pointed out,
the work in \cite{Be1} is a rediscovery of results of Sylvester \cite{Sy3} and Glaisher \cite{Gl};
see the discussion in Section \ref{sylvv}.

From now on we let $\rho := e^{2\pi i h/k}$.  Replacing $2\pi i z$ in \eqref{key} by $z$ and then $z + 2\pi i h/k$ we  obtain
\begin{align}\label{cabr1}
C_{hk\ell}(N) & =\operatornamewithlimits{Res}_{z=2\pi i h/k}  \frac{e^z(e^z-\rho)^{\ell-1}}{(1-e^z)(1-e^{2z}) \cdots (1-e^{Nz})}
\\
 & = (-1)^N \rho^\ell \operatornamewithlimits{Res}_{z=0}  e^z(e^z-1)^{\ell-1}
\left(\frac{1}{\rho e^z - 1}\right)\left(\frac{1}{\rho^2 e^{2z} - 1}\right) \cdots
\left(\frac{1}{\rho^N e^{N z} - 1}\right) \label{res0} \\
   & = \frac{(-1)^N \rho^\ell}{N!} \left[ \text{coeff. of }z^{N-1} \right]e^z (e^z-1)^{\ell-1}
\left(\frac{z}{\rho e^z - 1}\right)\left(\frac{2z
}{\rho^2 e^{2z} - 1}\right) \cdots \left(\frac{N z}{\rho^N e^{Nz} - 1}\right). \label{cabraaa}
\end{align}
Then \eqref{cabraaa} can be made more explicit by inserting the relevant power series. For the remainder of this section we focus on $h/k=0/1$ so that $\rho=1$.

Equation \eqref{cabraaa} implies
\begin{equation} \label{c000}
    C_{01\ell}(N) = \frac{(-1)^N}{N!} \left[ \text{coeff. of }z^{N-\ell} \right]e^z \left(\frac{e^z - 1}{z}\right)^{\ell-1}
\left(\frac{z}{e^z - 1}\right)\left(\frac{2z
}{e^{2z} - 1}\right) \cdots \left(\frac{N z}{e^{Nz} - 1}\right).
\end{equation}
Recall the well-known power series
\begin{align}\label{bern}
    \frac{z e^{tz}}{e^z-1} & = \sum_{n=0}^\infty B_n(t) \frac{z^n}{n!}, \qquad (|z|<2\pi)\\
    \left(\frac{e^z - 1}{z}\right)^m & = m! \sum_{n=0}^\infty \stirb{m+n}{m} \frac{z^n}{(m+n)!} \qquad (m \gqs 0)\label{stb}
\end{align}
with $B_n(t)$ the $n$th Bernoulli polynomial, $B_n:=B_n(0)$ the $n$th Bernoulli number and $\stirb{n}{m}$ the Stirling number,  denoting the number of ways to partition a set of size $n$ into $m$  non-empty subsets (see \cite[Eq. (7.49)]{Knu} for \eqref{stb}). 

\begin{prop} \label{cr}
We have
\begin{align}
    C_{01\ell}(N) & = \frac{(-1)^N (\ell-1)!}{ N!} \sum_{i+j_0+j_1+j_2+ \cdots + j_N = N-\ell}
    \stirb{\ell-1+j_0}{\ell-1} \frac{B_{j_1}B_{j_2}  \cdots  B_{j_N}}{i! (\ell-1+j_0)!}\frac{1^{j_1} 2^{j_2} \cdots
N^{j_N}}{ j_1 ! j_2 ! \cdots j_N!} \label{c01ra}\\
& = \frac{(-1)^N (\ell-1)!}{ N!} \sum_{j_0+j_1+j_2+ \cdots + j_N = N-\ell}
    \stirb{\ell-1+j_0}{\ell-1} \frac{B_{j_1}B_{j_2}  \cdots  B_{j_N}}{(\ell-1+j_0)!}\frac{ (-1)^{j_1} 2^{j_2} \cdots
N^{j_N}}{ j_1 ! j_2 ! \cdots j_N!}. \label{c01r}
\end{align}
\end{prop}
\begin{proof}
It is clear, using \eqref{bern} with $t=0$, that
\begin{equation*}
    \left(\frac{z}{e^z - 1}\right)\left(\frac{2z
}{e^{2z} - 1}\right) \cdots \left(\frac{N z}{e^{Nz} - 1}\right) = \sum_{r=0}^\infty z^r \sum_{j_1+j_2+ \cdots + j_N = r}
      B_{j_1}B_{j_2}  \cdots  B_{j_N} \frac{ 1^{j_1} 2^{j_2} \cdots
N^{j_N}}{j_1 ! j_2 ! \cdots j_N!}.
\end{equation*}
Combining this expansion with the power series for $e^z$ ($=\sum_i z^i/i!$) and $\left(\frac{e^z - 1}{z}\right)^{\ell-1}$ given by \eqref{stb} yields
\eqref{c01ra}.
For \eqref{c01r}, combine $e^z$  and $\left(\frac{z}{e^z - 1}\right)$ noting that
\begin{equation*}
    \frac{z e^z}{e^z-1} = \frac{-z}{e^{-z}-1} = \sum_{r=0}^\infty (-1)^r B_r \frac{z^r}{r!}.
\end{equation*}
Of course the coefficients $(-1)^r B_r$ above must equal $B_r(1)$ by \eqref{bern}, so we can replace the factor $B_{j_1} \cdot (-1)^{j_1}$ in \eqref{c01r} by the neater $B_{j_1}(1) \cdot 1^{j_1}$.
\end{proof}

For a perhaps more natural treatment, rewrite \eqref{c000} as
\begin{equation} \label{c100}
    C_{01\ell}(N) = \frac{(-1)^N}{N!} \left[ \text{coeff. of }z^{N-\ell} \right]e^z \left(\frac{z}{e^z - 1}\right)^{1-\ell}
\left(\frac{z}{e^z - 1}\right)\left(\frac{2z
}{e^{2z} - 1}\right) \cdots \left(\frac{N z}{e^{Nz} - 1}\right).
\end{equation}
We need a generalization of \eqref{bern}. The  Bernoulli polynomials of order $a$, denoted $B_n^{(a)}(t)$,  are defined by
\begin{equation*}
    \left(\frac{z }{e^z-1}\right)^a e^{tz} = \sum_{n=0}^\infty B_n^{(a)}(t) \frac{z^n}{n!}, \qquad (|z|<2\pi)
\end{equation*}
as in \cite[Eq. (24.16.1)]{DLMF}. It may be shown that $B_n^{(a)}(t)$ is a polynomial of degree $n$ in the variables $a$, $t$.
For $a=1$, $B_n^{(1)}(t)$ reduces to the usual Bernoulli polynomial $B_n(t)$. For $t=0$, $B_n^{(a)}:=B_n^{(a)}(0)$ is the N\"orlund polynomial in $a$.

\begin{prop} \label{crr}
We have
\begin{align}
C_{01\ell}(N) & = \frac{(-1)^N}{ N!} \sum_{j_0+j_1+j_2+ \cdots + j_N = N-\ell}
    B_{j_0}^{(1-\ell)}(1) \cdot B_{j_1}B_{j_2}  \cdots  B_{j_N} \frac{1^{j_1} 2^{j_2} \cdots
N^{j_N}}{j_0 ! j_1 ! j_2 ! \cdots j_N!} \label{c01rd}\\
& = \frac{(-1)^N (\ell-1)!}{ N!} \sum_{j_0+j_1+j_2+ \cdots + j_N = N-\ell}
     \stirb{\ell+j_0}{\ell}\frac{B_{j_1}B_{j_2}  \cdots  B_{j_N}}{(\ell-1+j_0)!}\frac{ 1^{j_1} 2^{j_2} \cdots
N^{j_N}}{j_1 ! j_2 ! \cdots j_N!}. \label{c01rf}
\end{align}
\end{prop}
\begin{proof}
Equation \eqref{c01rd} is clear from \eqref{c100}. Then \eqref{c01rf} follows
with
\begin{equation} \label{bjl}
    \frac{B_{j}^{(1-\ell)}(1)}{j!}= \frac{(\ell-1)!}{(\ell-1+j)!}\stirb{\ell+j}{\ell}, \qquad (j \gqs 0, \ \ell \gqs 1).
\end{equation}
In turn, \eqref{bjl} is a consequence of the identities
\begin{alignat}{2}\label{how}
    a \cdot B_n^{(a+1)}(1) & = (a-n) B_n^{(a)}, \qquad & &(n \in \Z_{\gqs 0}, \ a \in \C)\\
    B_n^{(-r)} & = \stirb{r+n}{r}\Big/\binom{r+n}{r}, \qquad & &(n \in \Z_{\gqs 0}, r \in \Z_{\gqs 1}) \label{norm}
\end{alignat}
where \eqref{how} is from \cite[Eq. (2.17)]{Ho} and \eqref{norm} is  \cite[Eq. (3)]{Ca}.
\end{proof}

The formulas for $C_{01\ell}(N)$ in Propositions \ref{cr} and \ref{crr} simplify for $\ell =1$. For example,
since $\stirb{j}{0} = \delta_{0,j}$, \eqref{c01ra} and \eqref{c01rf} imply the formula \eqref{c011} we saw in the introduction.

 We may recast our results using  ideas from \cite{Sy2}, \cite[pp. 3-4]{Be1}.  The expansion
\begin{equation} \label{lg}
    \log \left(\frac{z }{e^{z}-1} \right) = \sum_{m=1}^\infty (-1)^{m-1} \frac{B_m z^m}{m \cdot m!}
\end{equation}
may be proved by differentiating both sides. Define
\begin{equation}\label{smn}
s_m(N):=1^m+2^m+ \cdots + N^m
\end{equation}
and, after exponentiating \eqref{lg}, we obtain, (similarly to \cite[Eq. (6)]{Be1}),
\begin{multline} \label{ezz}
    e^z \left(\frac{z}{e^z - 1}\right)^{1-\ell}
\left(\frac{z}{e^z - 1}\right)\left(\frac{2z
}{e^{2z} - 1}\right) \cdots \left(\frac{N z}{e^{Nz} - 1}\right)\\
=\exp\left(z+ \sum_{m=1}^\infty (-1)^{m-1} \frac{B_m z^m}{m \cdot m!}\left(1-\ell+1^m +2^m + \cdots + N^m \right)\right)\\
     =\sum_{r=0}^\infty z^r \sum_{j_0+1j_1+2j_2+ \cdots r j_r = r}
     \frac{1}{j_0! j_1! j_2! \cdots j_r!}
     \left(\frac{B_1}{1 \cdot 1!}\bigl(s_1(N)+1-\ell\bigr) \right)^{j_1}  \cdots \left(\frac{(-1)^{r-1}B_r}{r \cdot r!}\bigl(s_r(N)+1-\ell \bigr) \right)^{j_r}.
\end{multline}
Combining \eqref{ezz} with \eqref{c100} produces
\begin{prop} \label{cr2}
We have
\begin{multline*}
    C_{01\ell}(N) = \frac{(-1)^N }{ N!} \sum_{j_0+1j_1+2j_2+ \cdots + N j_{N} = N-\ell}
     \frac{1}{j_0! j_1! j_2! \cdots j_N!} \\
    \times
     \left(\frac{B_1}{1 \cdot 1!}\bigl(s_1(N)+1-\ell\bigr) \right)^{j_1}  \cdots \left(\frac{(-1)^{N-1}B_{N}}{N \cdot N!}\bigl(s_{N}(N)+1-\ell \bigr) \right)^{j_{N}}.
\end{multline*}
\end{prop}

(The indices $j_{N-\ell+1}, \dots ,j_{N}$ are included only to simplify the statement.) The method of proof of Proposition \ref{cr2} will be important in Section \ref{zel}.


\section{Formulas for general Rademacher coefficients} \label{next}

Following Apostol in \cite[Eq. (3.1)]{Ap},  write
\begin{equation}\label{apb}
\frac{z}{\rho e^z - 1} = \sum_{m=0}^\infty \beta_m(\rho) \frac{z^m}{m!} \qquad (\rho \in \C).
\end{equation}
Then the next result has a similar proof to Proposition \ref{crr}.
\begin{prop} \label{ch}
We have
\begin{align}
    C_{hk\ell}(N)
& = \frac{(-1)^N \rho^{\ell}}{ N!} \sum_{j_0+j_1+ \cdots + j_N = N-\ell}
B_{j_0}^{(1-\ell)}(1) \cdot \beta_{j_1}(\rho)\beta_{j_2}(\rho^{2})  \cdots
\beta_{j_N}(\rho^{N})
     \frac{ 1^{j_1} 2^{j_2} \cdots  N^{j_N}}{j_0 ! j_1 ! j_2 ! \cdots
j_N!}\label{form1a}\\
& = \frac{(-1)^N \rho^{\ell}(\ell-1)!}{ N!} \sum_{j_0+j_1+ \cdots + j_N = N-\ell}
    \stirb{\ell+j_0}{\ell} \frac{ \beta_{j_1}(\rho)\beta_{j_2}(\rho^{2})  \cdots
\beta_{j_N}(\rho^{N})}{(\ell-1+j_0)!}  \frac{1^{j_1} 2^{j_2} \cdots  N^{j_N}}{ j_1 ! j_2 ! \cdots
j_N!}.\label{form1b}
\end{align}
\end{prop}

Clearly,
\begin{equation}\label{betmbm}
\beta_m(1)=B_m \qquad (m \gqs 0)
\end{equation}
so that \eqref{form1a}, \eqref{form1b} reduce to \eqref{c01rd} and \eqref{c01rf} for $h/k=0/1$.
To understand the Apostol-Bernoulli coefficients $\beta_m(\xi)$ for all $\xi$ we next express them in terms of the more familiar Bernoulli polynomials. For good measure we see they also have natural expressions  in terms of Stirling numbers
and the Hurwitz zeta function
$$\zeta(s,\alpha):=\sum_{n\in \Z_{\gqs 0}} (n+\alpha)^{-s}.$$

\begin{prop}
Let $\xi \in \C$ and let $m \in \Z_{\gqs 0}$. We have
\begin{alignat}{2}
\beta_m(\xi) & = (-1)^{m-1} m \sum_{j=1}^{m}  \stirb{m}{j} \frac{(j-1)!}{(\xi - 1)^j}
, \qquad & &(\xi \neq 1), \label{psi1}\\
& = k^{m-1} \sum_{j=0}^{k-1} \xi^{j} B_m(j/k), \qquad & &(\xi^k=1, \ k \in \Z_{\gqs 1}) \label{psi2}.
\end{alignat}
Also, for $a/b \in \Q$ with $0<a/b<1$ and $m \gqs 2$,
\begin{align}
\beta_m(e^{2\pi i a/b}) &= -\frac{m!}{(2\pi i)^m} \sum_{n \in \Z} \frac{1}{(n-a/b)^m} \label{psi4}\\
& = -\frac{m!}{(2\pi i)^m}\left(\zeta(m,1-a/b)+(-1)^m \zeta(m,a/b) \right). \label{psi3}
\end{align}
\end{prop}
\begin{proof}
An induction argument using the basic relation
\begin{equation*}
    \stirb{n}{j-1} + j \stirb{n}{j}  = \stirb{n+1}{j}
\end{equation*}
 shows
$$
\frac{d^n}{dz^n} \frac{1}{\xi e^z-1} = (-1)^n \sum_{j=1}^{n+1}  \stirb{n+1}{j}
\frac{(j-1)!}{\left(\xi e^z-1 \right)^{j}}
$$
and \eqref{psi1} follows.

When $\xi^k=1$ we have
$$
e^{kz}-1 = (\xi e^z)^{k}-1 = (\xi e^z-1)\left( 1 + \xi e^z+ \cdots +(\xi e^z)^{k-2} +(\xi e^z)^{k-1}
\right).
$$
Therefore
\begin{align*}
  \frac{z}{\xi e^z - 1} & = \frac{z }{e^{k z}-1} \sum_{j=0}^{k-1} \xi^{j} e^{j z} \\
   & =  \frac{1}k \sum_{j=0}^{k-1} \xi^{j} \frac{ k z \cdot e^{(j/k) \cdot kz} }{e^{k z}-1}\\
   & =  \frac{1}k \sum_{j=0}^{k-1} \xi^{j} \sum_{m=0}^\infty B_m(j/k)k^m  \frac{z^m}{m!}
\end{align*}
and this implies \eqref{psi2}.

Replace $B_m(x)$ in \eqref{psi2} with its Fourier expansion \cite[Eq. (24.8.3)]{DLMF}
$$
B_m(x)=-\frac{m!}{(2\pi i)^m} \sum_{n \neq 0} \frac{e^{2\pi i n x}}{n^m} \qquad (m \gqs 2, \ x\in [0,1]).
$$
Computing the resulting sum over $j$ and rearranging shows \eqref{psi4} and \eqref{psi3}.
\end{proof}

Thus we have from \eqref{psi1}, when $\xi \neq 1$,
$$
  \beta_0(\xi) = 0, \qquad
  \beta_1(\xi) =  \frac{1}{\xi-1}, \qquad
  \beta_2(\xi) = \frac{-2}{\xi-1} + \frac{-2}{(\xi-1)^2}= \frac{-2\xi}{(\xi-1)^2}, \quad \text{ etc.}
$$
A formula   equivalent to \eqref{psi1} was given by Glaisher in \cite[\S 97]{Gl}, (in his notation $F_n(x)=-\beta_n(x)/(n \cdot n!)$ for $n \gqs 2$ and the difference $\Delta^j 0^n$ is $ j! \stirb{n}{j})$. Glaisher provides four more variations of \eqref{psi1} in \cite[\S 100]{Gl} with, for example, the second and fourth being
\begin{align}
    \beta_m(\xi) & =  m \sum_{j=1}^{m}  \stirb{m-1}{j-1} (-\xi)^{j-1} \frac{(j-1)!}{(\xi - 1)^{j}}
, \qquad (\xi \neq 1, \ m \gqs 0), \label{gf2}\\
 & =  -m \sum_{j=0}^{m-2}  \stirb{m-2}{j} \xi^{j} \frac{j! (j+\xi)}{(1-\xi)^{j+2}}
, \qquad (\xi \neq 1, \ m \gqs 2). \label{gf4}
\end{align}
Apostol's result in \cite[Eq. (3.7)]{Ap}  is \eqref{gf2} and he used the coefficients $\beta_m(\xi)$ to describe the Lerch zeta function at negative integers. As discussed in \cite[Sect. 4]{Boy}, the Eulerian polynomials are very closely related to $\beta_m(\xi)$.

With \eqref{psi1}, \eqref{psi2} we see that $\beta_m(\xi)$ is in the  field $\Q(\xi)$ and  the next result then follows from Proposition \ref{ch}.
\begin{prop} We have
$C_{hk\ell}(N) \in \Q(e^{2\pi i h/k})$.
\end{prop}
Thus we only expect $C_{hk\ell}(N)$ to be rational when $h/k=0/1$ or $1/2$. We have already seen  rational expressions for $C_{01\ell}(N)$ in
Section \ref{tw}. To write $C_{12\ell}(N)$ explicitly, use \eqref{psi2}, \eqref{psi3} to get
\begin{alignat}{2}
\beta_m(-1) & = (2^m-1) B_m, \qquad & & (m \gqs 0) \label{psb}\\
     &=-\frac{m!}{(2\pi i)^m} (1+(-1)^m) \cdot \zeta(m,1/2), \qquad & & (m \gqs 2) \notag
\end{alignat}
where for \eqref{psb} we used the identity $B_m(1/2)=(2^{1-m}-1)B_m$ from \cite[Eq. (22.4.27)]{DLMF}.
Combining \eqref{psb} with \eqref{form1b}, for example, we find
\begin{multline*}
  C_{12\ell}(N) =  \frac{ (-1)^{N+\ell}(\ell-1)!}{ N!} \sum_{j_0+j_1+ \cdots + j_N = N-\ell}
    \stirb{\ell+j_0}{\ell}  \\
    \times \frac{ B_{j_1}B_{j_2} \cdots
B_{j_N}}{(\ell-1+j_0)!} \frac{ 1^{j_1} 2^{j_2} \cdots  N^{j_N}}{j_1 ! j_2 ! \cdots
j_N!} (2^{j_1}-1)(2^{j_3}-1) \cdots.
\end{multline*}

More generally, \eqref{form1b} and \eqref{psi2} allow us to write
\begin{multline} \label{iuy}
    C_{hk\ell}(N)  = \frac{(-1)^N \rho^{\ell}(\ell-1)!}{ N!} \sum_{j_0+j_1+ \cdots + j_N = N-\ell}
    \stirb{\ell+j_0}{\ell} \frac{k^{-\ell-j_0}}{(\ell-1+j_0)!} \frac{ 1^{j_1} 2^{j_2} \cdots  N^{j_N}}{ j_1 ! j_2 ! \cdots
j_N!} \\
\times \sum_{0 \lqs i_1, i_2, \dots, i_N \lqs k-1} \rho^{1 i_1+2 i_2+ \dots +N i_N} B_{j_1}(i_1/k) \cdots B_{j_N}(i_N/k).
\end{multline}

 \section{ Sylvester's waves} \label{sylvv}
Many of the formulas developed in Sections \ref{tw}, \ref{next} are  similar to classical ones of Sylvester and Glaisher. In this section we summarize part of their work on restricted partitions, highlighting the close connections between Rademacher's coefficients and Sylvester's waves.

For fixed $a_1$, $a_2, \dots ,a_N \in \Z_{\gqs 1}$ and each integer $n$,  the general restricted partition problem asks how many solutions there are to
\begin{equation*}
    a_1 x_1+ a_2 x_2 + \dots +a_N x_N =n \qquad (x_1, x_2, \dots , x_N \in \Z_{\gqs 0}).
\end{equation*}
Sylvester studied this number of solutions in  \cite{Sy1, Sy2, Sy3}, for example, calling it the {\em denumerant} (or {\em quotity})   of $n$ with respect to $a_1, \dots ,a_N$. For simplicity, we focus on the special case where $a_1=1$, $a_2=2, \dots ,a_N=N$ so that the denumerant is $p_N(n)$. It is straightforward to extend the results stated here back to the original results of Sylvester and Glaisher  for the general case.

With  $k$, $N \in \Z_{\gqs 1}$  and $n \in \Z$, Sylvester defined the {\em $k$-th wave} as
\begin{equation} \label{wave}
    W_k(N,n):=\res_{z=0} \sum_\rho \frac{\rho^n e^{n z}}{(1-\rho^{-1} e^{-z})(1-\rho^{-2} e^{-2z}) \cdots (1-\rho^{-N} e^{-Nz})}
\end{equation}
where the sum is over all primitive $k$-th roots of unity $\rho$.

\begin{theorem}[Sylvester's Theorem\footnote{The slightly involved early history of this theorem is described in \cite[pp. 119-121, 124, 125, 134, 135]{Di} and \cite[p. 277]{Gl} }]  \label{syth}
For $N$ and $n$ in $\Z_{\gqs 1}$, we have
\begin{equation*}
    p_N(n)=\sum_{k=1}^N W_k(N,n).
\end{equation*}
\end{theorem}
\begin{proof}
We may give a short proof based on our previous work. Write $e(z)$ for $e^{2\pi i z}$ and recall from \eqref{res0} that
\begin{equation}\label{res0b}
    C_{hk\ell}(N) = e(h/k)^\ell \operatornamewithlimits{Res}_{z=0}  \frac{e^z(e^z-1)^{\ell-1}}
{(1-e(h/k) e^z)(1-e(h/k)^2 e^{2z}) \cdots
(1-e(h/k)^N e^{N z})}.
\end{equation}
The Rademacher coefficients $C_{hk\ell}(N)$ were originally defined for $1 \lqs \ell \lqs \lfloor N/k \rfloor$, but we see that their definition may be extended by \eqref{res0b} to all integers $\ell > \lfloor N/k \rfloor$ and for these $\ell$s we have $C_{hk\ell}(N)=0$. Inserting \eqref{res0b} into \eqref{eul2} yields
\begin{align}
    p_N(n) & = \sum_{\substack{0\leqslant h<k \leqslant N \\ (h,k)=1}}
\sum_{\ell=1}^{\infty} \left(
\operatornamewithlimits{Res}_{z=0}   \frac{e(h/k)^\ell e^z(e^z-1)^{\ell-1}}
{(1-e(h/k) e^z) \cdots
(1-e(h/k)^N e^{N z})}
\right) (-1)^\ell \binom{\ell-1+n}{\ell-1} e(h/k)^{-\ell-n}\notag \\
 & =
 \operatornamewithlimits{Res}_{z=0} \sum_{\substack{0\leqslant h<k \leqslant N \\ (h,k)=1}}
 \frac{ e(h/k)^{-n} e^z}
{(1-e(h/k) e^z) \cdots
(1-e(h/k)^N e^{N z})}
   \sum_{\ell=1}^{\infty} (-1)^\ell \binom{\ell-1+n}{\ell-1} (e^z-1)^{\ell-1}. \label{inn}
\end{align}
With
\begin{equation*}
    (1+x)^{-m}=\sum_{j=0}^\infty \binom{-m}{j} x^j = \sum_{j=0}^\infty (-1)^j \binom{m-1+j}{m-1} x^j,
\end{equation*}
by the binomial theorem, the inner sum in \eqref{inn} is recognized as $-e^{-(n+1)z}$. Hence
\begin{equation*}
    p_N(n) = \operatornamewithlimits{Res}_{z=0} \sum_{\substack{0\leqslant h<k \leqslant N \\ (h,k)=1}}
 \frac{ -e(h/k)^{-n} e^{-nz}}
{(1-e(h/k) e^z) \cdots
(1-e(h/k)^N e^{N z})}.
\end{equation*}
The theorem follows on replacing $z$ by $-z$ and writing $\rho = e(h/k)^{-1}$.
\end{proof}

The wave $W_k(N,n)$ has period $k$ in the sense that it is given by a polynomial in $n$ that depends only on $n \bmod k$. In other words,
there exist $k$ polynomials
$$
Q_{k,0}(N,x), \ Q_{k,1}(N,x), \dots  ,Q_{k,k-1}(N,x)  \in \Q[x]
$$
so that $W_k(N,n)=Q_{k,n \bmod k}(N, n)$. The first wave $W_1(N,n)$ is the simplest, being just a polynomial in $n$.

\begin{prop}[Sylvester \cite{Sy3}] \label{swv}
For $N \in \Z_{\gqs 1}$ and $n \in \Z$
\begin{equation*}
    W_1(N,n) = \frac{(-1)^{N-1}}{ N!} \left[ \text{\rm coeff. of }z^{N-1} \right] \exp\left(
     -\left(n+\frac{N(N+1)}4 \right)z - \sum_{m=2}^{N-1} \frac{B_{m} \cdot s_{m}(N)}{m \cdot m!} z^m
     \right).
\end{equation*}
\end{prop}

The proof  of a   result generalizing Proposition \ref{swv} is given below. Glaisher, in \cite[\S\S 19-30]{Gl}, gave  more direct expressions for $W_1$  using the Bernoulli polynomial expansions \eqref{bern} with arguments $t=0$, $1/2$, $1$  in \eqref{wave}. For example, with $t=1/2$, he found
\begin{equation*}
    W_1(N,n) = \frac{1}{ N!} \sum_{j_0+j_1+ \cdots + j_N = N-1}
    \left(n+\frac{N(N+1)}4 \right)^{j_0} B_{j_1}\left(\frac 12 \right)   \cdots  B_{j_N}\left(\frac 12 \right) \frac{ 1^{j_1}  \cdots
N^{j_N}}{ j_0 ! j_1 !  \cdots j_N!}.
\end{equation*}
In \cite[Theorem, p. 2]{Be1} the authors have essentially rediscovered the above  formulas of Sylvester and Glaisher  for the first wave $W_1$.

Glaisher supplies  detailed formulas for $W_2, \dots ,W_6$ in \cite[\S\S 31-76]{Gl}, including expressions in terms of circulators (periodic sequences of integers), see also \cite[Sect. 3]{An}.  In \cite[\S\S 88-100]{Gl} he treats the general wave of period $k$. We use the Apostol-Bernoulli coefficients $\beta_m(\rho)$ from \eqref{apb} to state his result.
For fixed $k$ and $0 \lqs r \lqs k-1$ set
\begin{equation*}
    s_{m,r}(N):= \sum_{\substack{1 \lqs j \lqs N , \ j \equiv r \bmod k}} j^m.
\end{equation*}

\begin{theorem}[Glaisher] \label{glai} For $k, N \in \Z_{\gqs 1}$, $s :=\lfloor N/k \rfloor$ and $n \in \Z$
\begin{multline*}
    W_k(N,n) = \sum_\rho \frac {(-1)^{N-1} \rho^{-n} }{ k^{s} \cdot s!  } \left(\prod_{ 1 \lqs w \lqs N, \ k \nmid w} \frac 1{\rho^w  -1} \right) \\
    \times
     \left[ \text{\rm coeff. of }z^{s-1} \right]
     \exp\left(-\left(n+\frac{N(N+1)}2 \right)z
     - \sum_{r=0}^{k-1} \sum_{m=1}^{s-1} \frac{ \beta_m(\rho^r) \cdot s_{m,r}(N)}{m \cdot m!} z^m \right).
\end{multline*}
\end{theorem}

\begin{proof}
Write
\begin{equation*}
    W_k(N,n)=\sum_\rho \res_{z=0} (-1)^{N-1} \rho^{-n} e^{-n z} \prod_{1 \lqs w \lqs N} \frac 1{\rho^w e^{w z} -1}.
\end{equation*}
The product equals
\begin{equation*}
    \prod_{\substack{ 1 \lqs w \lqs N \\ k \mid w}} \frac 1{w z} \cdot
    \prod_{\substack{ 1 \lqs w \lqs N \\ k \mid w}} \frac {wz}{ e^{w z} -1} \cdot
    \prod_{\substack{ 1 \lqs w \lqs N \\ k \nmid w}} \frac 1{\rho^w  -1} \cdot
    \prod_{\substack{ 1 \lqs w \lqs N \\ k \nmid w}} \frac {\rho^w  -1}{\rho^w e^{w z} -1}
\end{equation*}
and
hence
\begin{equation} \label{edf}
    W_k(N,n)=\sum_\rho \frac {(-1)^{N-1} \rho^{-n} }{ k^{s} \cdot s!  } \prod_{\substack{ 1 \lqs w \lqs N \\ k \nmid w}} \frac 1{\rho^w  -1}
     \left[ \text{coeff. of }z^{s-1} \right] e^{-n z} \cdot \prod_{\substack{ 1 \lqs w \lqs N \\ k \mid w}} \frac {wz}{ e^{w z} -1} \cdot
    \prod_{\substack{ 1 \lqs w \lqs N \\ k \nmid w}} \frac {\rho^w  -1}{\rho^w e^{w z} -1} .
\end{equation}
With \eqref{lg} and its analog
\begin{equation}\label{ana}
    \log \left( \frac{\rho -1}{\rho e^z -1} \right) = -z -\sum_{m=1}^\infty \frac{\beta_m(\rho)}{m \cdot m!} z^m \qquad (\rho \neq 1)
\end{equation}
we can write the log of the part of \eqref{edf} depending on $z$ as
\begin{multline*}
    -nz+ \sum_{\substack{ 1 \lqs w \lqs N \\ k \mid w}} \log \left( \frac {wz}{ e^{w z} -1} \right)
    + \sum_{\substack{ 1 \lqs w \lqs N \\ k \nmid w}} \log \left( \frac {\rho^w  -1}{\rho^w e^{w z} -1} \right)\\
    =
    -nz -  \sum_{\substack{ 1 \lqs w \lqs N \\ k \mid w}} \sum_{m=1}^\infty \frac{(-1)^{m} B_m}{m \cdot m!} w^m z^m
    - \sum_{\substack{ 1 \lqs w \lqs N \\ k \nmid w}}  \left( w z + \sum_{m=1}^\infty \frac{(-1)^{m} \beta_m(\rho^w)}{m \cdot m!} w^m z^m \right)\\
    =
    -nz -   \sum_{m=1}^\infty \frac{(-1)^{m} B_m \cdot s_{m,0}(N)}{m \cdot m!}  z^m
    - \sum_{\substack{ 1 \lqs r \lqs k-1}}  \left( s_{1,r}(N) z + \sum_{m=1}^\infty \frac{(-1)^{m} \beta_m(\rho^r) \cdot s_{m,r}(N)}{m \cdot m!} z^m \right).
\end{multline*}
Regrouping with  \eqref{betmbm}
yields the theorem.
\end{proof}

The product in the statement of Theorem \ref{glai} may be simplified. We need a  lemma.

\begin{lemma}\label{rou}
Let $\rho$ be a primitive $k$-th root of unity. Then
$$
(1-\rho)(1-\rho^2) \cdots (1-\rho^{k-1})=k.
$$
\end{lemma}
\begin{proof}
The zeros of the polynomial $p(x)=(x+1)^k-1$ are exactly $\rho^j-1$ for $0 \lqs j \lqs k-1$. Hence the zeros of
\begin{equation}\label{bp}
p(x)/x = x^{k-1}+\binom{k}{1}x^{k-2}+ \cdots +\binom{k}{k-2}x+\binom{k}{k-1}
\end{equation}
are $\rho^j-1$ for $1 \lqs j \lqs k-1$ and their product is $(-1)^{k-1}$ times the constant term in \eqref{bp}.
\end{proof}

It follows that
\begin{align}
    \prod_{\substack{ 1 \lqs w \lqs N \\ k \nmid w}} \frac 1{\rho^w  -1} & =
    \left( \prod_{w=1}^{k-1} \frac 1{\rho^w  -1} \right)^s \prod_{w=1}^{N-ks} \frac 1{\rho^w  -1} \notag \\
    & =
     \left(  \frac {(-1)^{k-1}}{k} \right)^s \frac{(-1)^{N-ks}}{\prod_{1 \lqs w \lqs N-ks} (1-\rho^w)  } \notag \\
    & =
    \frac {(-1)^{N+s}}{k^{s} \prod_{1 \lqs w \lqs N-ks} (1-\rho^w)  }. \label{poix}
\end{align}

Using \eqref{poix} in Theorem \ref{glai} and expanding allows us to write the $k$-th wave relatively transparently as
\begin{multline}
    W_k(N,n) = \sum_\rho \frac {(-1)^{s-1} \rho^{-n} }{k^{2s} \cdot s! \prod_{1 \lqs w \lqs N-ks} (1-\rho^w)  } \\
    \times
     \sum_{1j_1+2j_2+ \cdots + N j_{N} = s-1}
     \frac{1}{j_1! j_2! \cdots j_N!}
     \left(-n-\frac{N(N+1)}2 - \sum_{r=0}^{k-1}  \frac{ \beta_1(\rho^r) \cdot s_{1,r}(N)}{1 \cdot 1!}\right)^{j_1} \\
    \times
     \left(- \sum_{r=0}^{k-1}  \frac{ \beta_2(\rho^r) \cdot s_{2,r}(N)}{2 \cdot 2!} \right)^{j_{2}} \cdots
     \left(- \sum_{r=0}^{k-1}  \frac{ \beta_N(\rho^r) \cdot s_{N,r}(N)}{N \cdot N!} \right)^{j_{N}}. \label{wavek}
\end{multline}
Putting $k=1$ in \eqref{wavek} gives an expanded version of Proposition \ref{swv}.

We conclude this section by noting the following explicit relations, which the reader may readily verify:
\begin{align*}
   W_k(N,n) & = -\sum_{\substack{0\leqslant h<k \\ (h,k)=1}}
\sum_{\ell=1}^{\lfloor N/k \rfloor}  \binom{-n-1}{\ell-1} e(h/k)^{-\ell-n} C_{hk\ell}(N) \qquad (n \in \Z),\\
C_{01\ell}(N) & = \sum_{j=1}^\ell \binom{\ell -1}{j-1} (-1)^{j} W_1(N,-j), \\
C_{12\ell}(N) & = \sum_{j=1}^\ell \binom{\ell -1}{j-1} (-1)^{\ell} W_2(N,-j).
\end{align*}

\section{A recursive formula for $C_{hk\ell}(N)$} \label{s4}
Our aim in this section is to give a recursive form of \eqref{form1b} and \eqref{iuy}, useful for computations.
For $\rho = e^{2\pi i h/k}$ as before and $N,m \gqs 1$,  set
\begin{equation} \label{dhkl}
    D_{hk\ell}(N,m) :=  \sum_{j_0+j_1+ \cdots + j_N = m-\ell}
    \stirb{\ell+j_0}{\ell} \frac{ \beta_{j_1}(\rho)\beta_{j_2}(\rho^{2})  \cdots
\beta_{j_N}(\rho^{N})}{(\ell-1+j_0)!} \frac{ 1^{j_1} 2^{j_2} \cdots  N^{j_N}}{ j_1 ! j_2 ! \cdots
j_N!}.
\end{equation}
As a consequence of \eqref{form1b},
\begin{equation*}
    C_{hk\ell}(N) = \frac{(-1)^N (\ell-1)!}{ N!} \rho^{\ell} D_{hk\ell}(N,N).
\end{equation*}
We may write
\begin{equation} \label{dab}
    D_{hk\ell}(N,m) = \sum_{r=0}^{k-1} \rho^r E_{hk\ell}(N,m;r)
\end{equation}
for $E_{hk\ell}(N,m;r) \in \Q$ and we wish to find a recursive formula for these rational numbers.

First note that
\begin{equation}\label{stp0}
E_{hk\ell}(0,m;r) =  \delta_{0,r} \stirb{m}{\ell}\frac{1}{(m-1)!}
\end{equation}
follows from \eqref{dhkl}. Substituting formulas \eqref{dab}, \eqref{psi2} into both sides of the identity
$$
D_{hk\ell}(N,m) =\sum_{a=0}^{m-\ell} D_{hk\ell}(N-1,m-a) \frac{\beta_a(\rho^N) N^a}{a!} \qquad (N \gqs 1)
$$
and equating coefficients of $\rho^r$ produces:
\begin{equation} \label{stp1}
    E_{hk\ell}(N,m;r) = \sum_{a=0}^{m-\ell} N^a \sum_{j=0}^{k-1} E_{hk\ell}(N-1,m-a;(r-N j) \bmod k ) \frac{k^{a-1} }{a!} B_a(j/k).
\end{equation}
By induction, it is now clear from \eqref{stp0} and \eqref{stp1} that $E_{hk\ell}(N,m;r)$ is independent of $h$ and so we denote it just $E_{k\ell}(N,m;r)$.
We have proved the following.

\begin{theorem} \label{cpu}
For $N, m \gqs 1$ and $0\lqs r \lqs k-1$, define  recursively
\begin{align*}
    E_{k\ell}(0,m;r) & := \delta_{0,r} \stirb{m}{\ell}\frac{1}{(m-1)!},\\
    E_{k\ell}(N,m;r) & := \sum_{a=0}^{m-\ell} N^a \sum_{j=0}^{k-1} E_{k\ell}(N-1,m-a;(r- N j) \bmod k ) \frac{k^{a-1} }{a!} B_a(j/k) \qquad (N \gqs 1).
\end{align*}
With $\rho = e^{2\pi i h/k}$ we then have
\begin{equation} \label{ckk}
     C_{hk\ell}(N) = \frac{(-1)^N (\ell-1)!}{ N!} \sum_{r=0}^{k-1} \rho^{r+\ell} E_{k\ell}(N,N;r).
\end{equation}
\end{theorem}

Theorem \ref{cpu} allows us to calculate $C_{hk\ell}(N)$ quickly, just using rational numbers until the final step. All the computations in the  next section were carried out by this method.

We also remark that Theorem \ref{cpu} allows us to easily sum $ C_{hk\ell}(N)$ over all $h$ prime to $k$. Recall from \cite[Eq. (3.1)]{IwKo}, for example, the identity
\begin{equation} \label{rsum}
    \sum_{\substack{0\lqs a <b \\ (a,b)=1}} e^{2\pi i m a/b} = \sum_{d | (m,b)} d \cdot \mu(b/d)
\end{equation}
where the left side of \eqref{rsum}  is  a {\em Ramanujan sum} and the right side has the M\"obius function $\mu$. Therefore \eqref{ckk} and \eqref{rsum} imply
\begin{equation*}
    \sum_{\substack{0\lqs h <k \\ (h,k)=1}}C_{hk\ell}(N) = \frac{(-1)^N (\ell-1)!}{ N!} \sum_{r=0}^{k-1}  E_{k\ell}(N,N;r) \sum_{d | (r+\ell,k)} d \cdot \mu(k/d),
\end{equation*}
a rational number.

\section{Rademacher's conjecture}\label{rcoj}
Rademacher  modified the method of Hardy and Ramanujan to find an exact formula in \cite[Eq. (4)]{Ra2} for $p(n)$, the number of  partitions of $n$. Also in this 1937 paper he  substituted his  formula back into
\begin{equation}\label{rde}
\sum_{n=0}^\infty p(n) q^n = \prod_{j=1}^\infty \frac{1}{1-q^j}
\end{equation}
to obtain in \cite[Eqs. (13), (14)]{Ra2} a decomposition of the right side of \eqref{rde} into partial fractions.
This is detailed in \cite[pp. 292 - 302]{Ra} and Rademacher finds
\begin{equation}\label{37}
\prod_{j=1}^\infty \frac{1}{1-q^j}=\sum_{\substack{0\leqslant h<k  \\ (h,k)=1}}
\sum_{\ell=1}^{\infty} \frac{C_{hk\ell}(\infty)}{(q-e^{2\pi ih/k})^\ell} \qquad(|q|<1)
\end{equation}
 with coefficients $C_{hk\ell}(\infty)$ given explicitly in \cite[Eq. (130.6)]{Ra}.
For example
$$
C_{011}(\infty) = -\frac 6{25} -\frac{12 \sqrt{3}}{125 \pi}, \qquad C_{121}(\infty)  = \frac{\sqrt{3}-3}{25} +\frac{12 (\sqrt{3}+3)}{125 \pi}.
$$

Comparing \eqref{37} with \eqref{tp}, he then proposed the following appealing conjecture, providing some limited numerical evidence with $N \lqs 5$.
\begin{conj}[Rademacher \cite{Ra}] \label{rademacher}
We have
$$\lim_{N \rightarrow \infty} C_{hk\ell}(N)=C_{hk\ell}(\infty).$$
\end{conj}
Andrews redrew attention to Conjecture \ref{rademacher} in \cite{An} and subsequently  Munagi \cite{Mu2,Mu} and Davidson and Gagola in \cite{DG} considered the problem, though without making headway on the original conjecture. In \cite{DG} they calculated $C_{011}(N)$ with $N \leqslant 45$. These values showed oscillation and the
authors seem unconvinced that the sequence is converging. Using a new recursive technique  (see Corollary \ref{zco}), Sills and Zeilberger in \cite{SZ} were able to compute a much larger range of $C_{011}(N)$, showing clearly that this sequence looks to be  oscillating with period close to 32 and with amplitude growing exponentially.

In work currently being completed in \cite{OS}, we extend the techniques of Sections \ref{tw} -- \ref{s4} and employ the saddle-point method to obtain precise asymptotics for   $C_{011}(N)$ as $N \to \infty$. This shows an interesting  link with the zeros of the dilogarithm that we describe next.

Recall that the dilogarithm $\li(z)$ is initially defined as
$$
\li(z):=\sum_{m=1}^\infty \frac{z^m}{m^2} \qquad (|z|\lqs 1).
$$
It  has an analytic continuation to all $z \in \C$ by means of
\begin{equation}\label{lii}
\li(z) = -\int_0^z \log(1-u) \frac{du}{u}
\end{equation}
with a branch point at $z=1$. It may be shown that
$$
\li(w)+2\pi i \log(w) = 0
$$
 has a unique solution for $w \in \C$  given by
 $w_0 \approx 0.916198 + 0.182459 i$. In fact $w_0$ is a zero of the dilogarithm on a non-principal branch because, as the contour of integration in \eqref{lii} passes down across the branch cut $[1,\infty)$, the term $2\pi i \log(z)$ gets added to the principal value, as in \cite[Sect. 3(b)]{max}. It is convenient to set
$
z_0:=\log(1-w_0)/(-2\pi i)+1$
so that
$$
w_0=1-e^{-2\pi i z_0}, \quad 1<\Re(z_0)<2.
$$

\begin{conj} \label{coonj1}
We have
\begin{equation}\label{cj1}
   C_{011}(N)=\Re\left[(-2  z_0 e^{\pi i z_0})\frac{w_0^{-N}}{N^2}\right] +O\left( \frac{|w_0|^{-N}}{N^3}\right).
\end{equation}
\end{conj}
Presenting \eqref{cj1} with real numbers, we can equivalently write
\begin{equation} \label{cj1real}
    C_{011}(N)= \frac{ e^{N U}}{N^2}\left( \alpha \sin(\beta +N V) +O\left( \frac 1{N}\right)\right)
\end{equation}
for
\begin{equation*}
    \alpha \approx 5.39532, \quad \beta \approx 1.92792, \quad  U \approx 0.0680762, \quad V \approx -0.196576.
\end{equation*}
Thus, Conjecture \ref{coonj1} gives an exact version of  \cite[Conjecture 2.1]{SZ}. (In fact \eqref{cj1} and \eqref{cj1real} can be made even more  precise with a complete asymptotic expansion.) The period of the oscillations on the right of \eqref{cj1}, \eqref{cj1real} is $-2 \pi/V \approx 31.9631$.

\begin{conj} \label{coonj2}
We have
\begin{equation} \label{cj2}
    C_{121}(N)=\Re\left[
    -z_0 \sqrt{2 e^{\pi i z_0} \left(e^{\pi i z_0} + (-1)^N \right)} \frac{w_0^{-N/2}}{N^2}
    \right] +O\left( \frac{|w_0|^{-N/2}}{N^3}\right).
\end{equation}
\end{conj}

The real number version of Conjecture \ref{coonj2} is the same as \eqref{cj1real} except that $U$ and $V$ are substituted by $U/2$ and $V/2$. Also $\alpha$ and $\beta$ are replaced by $\alpha' \approx 4.51129$, $\beta' \approx  -1.30059$ if $N$ is odd and replaced by $\alpha'' \approx 3.11832$, $\beta'' \approx  -1.02847$ if $N$ is even.

\begin{table}[h]
\begin{center}
\begin{tabular}{c||c|c||c|c}
$N$ & $C_{011}(N)$ & $A_{011}(N)$ & $C_{121}(N)$ & $A_{121}(N)$ \\ \hline
$200$ & $32.1168$ & $33.8689$ & $0.0253518 $ & $-0.0680541 $ \\
$400$ & $-2.16712 \times 10^7$ & $-2.17937 \times 10^7$ & $ -7.89072$ & $-7.60602 $ \\
$600$ & $-1.77255 \times 10^{12}$ & $-1.80284 \times 10^{12}$ & $1838.23 $ & $1963.12 $ \\
$800$ & $3.71444 \times 10^{18}$ & $3.72536 \times 10^{18}$ & $2.91228 \times 10^6 $ & $2.93686 \times 10^6 $\\
$1000$ & $2.54070 \times 10^{23}$ & $2.58000 \times 10^{23}$ & $1.77778  \times 10^9$ & $1.7713 \times 10^9 $
\end{tabular}
\caption{} \label{evid}
\end{center}
\end{table}

Numerical evidence for Conjectures \ref{coonj1}, \ref{coonj2} is shown in Table \ref{evid} with $A_{011}(N)$ and $A_{121}(N)$ denoting the main terms on the right  of \eqref{cj1}, \eqref{cj2}. Figures \ref{pfig} and \ref{qfig} contain more verifying data.



\SpecialCoor
\psset{griddots=5,subgriddiv=0,gridlabels=0pt}
\psset{xunit=0.05cm, yunit=0.4cm}
\psset{linewidth=1pt}
\psset{dotsize=2pt 0,dotstyle=*}

\begin{figure}[h]
\begin{center}
\begin{pspicture}(80,-7)(340,6) 

\savedata{\mydata}[
{{100, 4.85601}, {101, 4.30324}, {102, 3.58473}, {103, 2.72814}, {104,
   1.76646}, {105,
  0.736751}, {106, -0.321342}, {107, -1.36706}, {108, -2.36011},
{109, -3.26227}, {110, -4.03876}, {111, -4.65969}, {112, -5.10114},
{113, -5.34611}, {114, -5.38515}, {115, -5.21677}, {116, -4.84745},
{117, -4.29142}, {118, -3.57009}, {119, -2.71124}, {120, -1.74797},
{121, -0.717369}, {122, 0.340865}, {123, 1.38597}, {124,
  2.37769}, {125, 3.27782}, {126, 4.05171}, {127, 4.66952}, {128,
  5.10748}, {129, 5.34871}, {130, 5.38392}, {131, 5.21175}, {132,
  4.83883}, {133, 4.27953}, {134, 3.5554}, {135, 2.69432}, {136,
  1.72945}, {137,
  0.697978}, {138, -0.360383}, {139, -1.40486}, {140, -2.39523},
{141, -3.29334}, {142, -4.0646}, {143, -4.67929}, {144, -5.11375},
{145, -5.35124}, {146, -5.38261}, {147, -5.20665}, {148, -4.83015},
{149, -4.26759}, {150, -3.54066}, {151, -2.67735}, {152, -1.71092},
{153, -0.678578}, {154, 0.379897}, {155, 1.42374}, {156,
  2.41274}, {157, 3.30881}, {158, 4.07743}, {159, 4.689}, {160,
  5.11995}, {161, 5.3537}, {162, 5.38123}, {163, 5.20149}, {164,
  4.8214}, {165, 4.2556}, {166, 3.52588}, {167, 2.66035}, {168,
  1.69235}, {169,
  0.659169}, {170, -0.399406}, {171, -1.4426}, {172, -2.43022}, {173,
-3.32424}, {174, -4.09021}, {175, -4.69864}, {176, -5.12609}, {177,
-5.35609}, {178, -5.37979}, {179, -5.19626}, {180, -4.81259}, {181,
-4.24355}, {182, -3.51105}, {183, -2.64332}, {184, -1.67377}, {185,
-0.639751}, {186, 0.41891}, {187, 1.46144}, {188, 2.44767}, {189,
  3.33962}, {190, 4.10294}, {191, 4.70823}, {192, 5.13216}, {193,
  5.35841}, {194, 5.37827}, {195, 5.19096}, {196, 4.80372}, {197,
  4.23144}, {198, 3.49618}, {199, 2.62625}, {200, 1.65516}, {201,
  0.620325}, {202, -0.438408}, {203, -1.48025}, {204, -2.46508},
{205, -3.35496}, {206, -4.11562}, {207, -4.71775}, {208, -5.13816},
{209, -5.36066}, {210, -5.37668}, {211, -5.1856}, {212, -4.79478},
{213, -4.21928}, {214, -3.48126}, {215, -2.60915}, {216, -1.63654},
{217, -0.60089}, {218, 0.4579}, {219, 1.49905}, {220, 2.48247}, {221,
  3.37026}, {222, 4.12824}, {223, 4.72721}, {224, 5.14409}, {225,
  5.36284}, {226, 5.37502}, {227, 5.18016}, {228, 4.78578}, {229,
  4.20706}, {230, 3.46629}, {231, 2.59201}, {232, 1.61789}, {233,
  0.581448}, {234, -0.477387}, {235, -1.51783}, {236, -2.49982},
{237, -3.38551}, {238, -4.14081}, {239, -4.7366}, {240, -5.14996},
{241, -5.36494}, {242, -5.37329}, {243, -5.17466}, {244, -4.77672},
{245, -4.19479}, {246, -3.45128}, {247, -2.57484}, {248, -1.59922},
{249, -0.561998}, {250, 0.496867}, {251, 1.53659}, {252,
  2.51713}, {253, 3.40072}, {254, 4.15332}, {255, 4.74594}, {256,
  5.15575}, {257, 5.36698}, {258, 5.37148}, {259, 5.16909}, {260,
  4.76759}, {261, 4.18246}, {262, 3.43622}, {263, 2.55763}, {264,
  1.58052}, {265,
  0.542541}, {266, -0.516341}, {267, -1.55533}, {268, -2.53442},
{269, -3.41588}, {270, -4.16578}, {271, -4.75521}, {272, -5.16148},
{273, -5.36895}, {274, -5.36961}, {275, -5.16345}, {276, -4.7584},
{277, -4.17007}, {278, -3.42112}, {279, -2.54039}, {280, -1.56181},
{281, -0.523076}, {282, 0.535807}, {283, 1.57405}, {284,
  2.55167}, {285, 3.431}, {286, 4.17818}, {287, 4.76442}, {288,
  5.16715}, {289, 5.37085}, {290, 5.36767}, {291, 5.15774}, {292,
  4.74915}, {293, 4.15763}, {294, 3.40597}, {295, 2.52312}, {296,
  1.54308}, {297,
  0.503605}, {298, -0.555267}, {299, -1.59275}, {300, -2.56889}}]
\dataplot[linecolor=red,linewidth=0.8pt,plotstyle=curve]{\mydata}

\savedata{\mydata}[
{{100, 1.42646}, {101, 1.04293}, {102,
  0.493117}, {103, -0.195683}, {104, -0.991078}, {105, -1.85682},
{106, -2.75419}, {107, -3.64348}, {108, -4.48552}, {109, -5.24318},
{110, -5.88277}, {111, -6.37536}, {112, -6.6979}, {113, -6.83405},
{114, -6.77487}, {115, -6.5191}, {116, -6.07325}, {117, -5.45132},
{118, -4.67428}, {119, -3.76922}, {120, -2.76835}, {121, -1.7077},
{122, -0.625773}, {123, 0.437961}, {124, 1.4446}, {125,
  2.35731}, {126, 3.14274}, {127, 3.77235}, {128, 4.22346}, {129,
  4.48021}, {130, 4.53409}, {131, 4.38436}, {132, 4.03805}, {133,
  3.50966}, {134, 2.8207}, {135, 1.99879}, {136, 1.07662}, {137,
  0.0907375}, {138, -0.919915}, {139, -1.91545}, {140, -2.85659},
{141, -3.70618}, {142, -4.43059}, {143, -5.00105}, {144, -5.39472},
{145, -5.59562}, {146, -5.59517}, {147, -5.3926}, {148, -4.99494},
{149, -4.41676}, {150, -3.67962}, {151, -2.81121}, {152, -1.84434},
{153, -0.815614}, {154, 0.235917}, {155, 1.2703}, {156,
  2.24821}, {157, 3.13246}, {158, 3.88941}, {159, 4.49033}, {160,
  4.91243}, {161, 5.1398}, {162, 5.16398}, {163, 4.98433}, {164,
  4.60803}, {165, 4.04981}, {166, 3.33139}, {167, 2.48065}, {168,
  1.53055}, {169,
  0.51787}, {170, -0.518209}, {171, -1.53761}, {172, -2.5009}, {173,
-3.37081}, {174, -4.11368}, {175, -4.70074}, {176, -5.1092}, {177,
-5.32319}, {178, -5.33432}, {179, -5.142}, {180, -4.7535}, {181,
-4.18364}, {182, -3.45425}, {183, -2.59331}, {184, -1.63384}, {185,
-0.612723}, {186, 0.430817}, {187, 1.45666}, {188, 2.42537}, {189,
  3.29967}, {190, 4.04596}, {191, 4.63551}, {192, 5.04565}, {193,
  5.2606}, {194, 5.27208}, {195, 5.07967}, {196, 4.69078}, {197,
  4.12038}, {198, 3.39046}, {199, 2.52911}, {200, 1.56954}, {201,
  0.548697}, {202, -0.494078}, {203, -1.51861}, {204, -2.4854}, {205,
-3.35719}, {206, -4.10036}, {207, -4.68625}, {208, -5.09225}, {209,
-5.30266}, {210, -5.30932}, {211, -5.11193}, {212, -4.71802}, {213,
-4.14272}, {214, -3.40813}, {215, -2.54249}, {216, -1.57909}, {217,
-0.555}, {218, 0.490375}, {219, 1.5168}, {220, 2.48476}, {221,
  3.357}, {222, 4.09992}, {223, 4.68492}, {224, 5.08946}, {225,
  5.29794}, {226, 5.30233}, {227, 5.10244}, {228, 4.70595}, {229,
  4.12811}, {230, 3.39115}, {231, 2.52345}, {232, 1.5584}, {233,
  0.533154}, {234, -0.512804}, {235, -1.5392}, {236, -2.50649}, {237,
-3.37744}, {238, -4.11848}, {239, -4.70106}, {240, -5.10272}, {241,
-5.30799}, {242, -5.30892}, {243, -5.10547}, {244, -4.70544}, {245,
-4.12421}, {246, -3.38415}, {247, -2.51374}, {248, -1.54648}, {249,
-0.519606}, {250, 0.527338}, {251, 1.55404}, {252, 2.52097}, {253,
  3.39087}, {254, 4.13025}, {255, 4.71061}, {256, 5.1096}, {257,
  5.31184}, {258, 5.30951}, {259, 5.10269}, {260, 4.69932}, {261,
  4.11494}, {262, 3.37201}, {263, 2.49914}, {264, 1.52994}, {265,
  0.501717}, {266, -0.545931}, {267, -1.57266}, {268, -2.53893},
{269, -3.40752}, {270, -4.14497}, {271, -4.72286}, {272, -5.11894},
{273, -5.31793}, {274, -5.31215}, {275, -5.1018}, {276, -4.69497},
{277, -4.10731}, {278, -3.36143}, {279, -2.48605}, {280, -1.51487},
{281, -0.485276}, {282, 0.563081}, {283, 1.58983}, {284,
  2.55543}, {285, 3.42269}, {286, 4.15819}, {287, 4.73362}, {288,
  5.12679}, {289, 5.32255}, {290, 5.31335}, {291, 5.09952}, {292,
  4.68928}, {293, 4.09843}, {294, 3.3497}, {295, 2.47191}, {296,
  1.49887}, {297,
  0.468037}, {298, -0.580894}, {299, -1.60753}, {300, -2.57232}}
]
\dataplot[linecolor=black,linewidth=0.8pt,plotstyle=dots]{\mydata}
\psaxes[linecolor=gray,Ox=100,Oy=0,Dx=50,dx=50,Dy=2,dy=2]{->}(100,0)(80,-7)(320,6)

  \rput(320,-1){$N$}
  \rput(320,5){$\alpha \sin(\beta+NV)$}

\end{pspicture}
\caption{$C_{011}(N)*N^2 e^{-NU}$ for $100 \leqslant N \leqslant 300$ \label{pfig}}
\end{center}
\end{figure}




\SpecialCoor
\psset{griddots=5,subgriddiv=0,gridlabels=0pt}
\psset{xunit=0.04cm, yunit=0.5cm}
\psset{linewidth=1pt}
\psset{dotsize=2pt 0,dotstyle=*}

\begin{figure}[h]
\begin{center}
\begin{pspicture}(230,-5)(560,5) 

\savedata{\mydata}[
{{250, -1.40597}, {251, -1.67231}, {252, -1.92252}, {253, -2.15417},
{254, -2.36502}, {255, -2.55304}, {256, -2.71642}, {257, -2.85358},
{258, -2.9632}, {259, -3.04421}, {260, -3.09583}, {261, -3.11758},
{262, -3.10923}, {263, -3.07086}, {264, -3.00286}, {265, -2.90587},
{266, -2.78082}, {267, -2.62894}, {268, -2.45168}, {269, -2.25075},
{270, -2.0281}, {271, -1.78587}, {272, -1.52641}, {273, -1.25221},
{274, -0.965918}, {275, -0.670306}, {276, -0.368223}, {277,
-0.0625862}, {278, 0.243655}, {279, 0.547544}, {280, 0.846148}, {281,
  1.13658}, {282, 1.41605}, {283, 1.68184}, {284, 1.93141}, {285,
  2.16232}, {286, 2.37237}, {287, 2.55952}, {288, 2.72196}, {289,
  2.85812}, {290, 2.9667}, {291, 3.04664}, {292, 3.09717}, {293,
  3.1178}, {294, 3.10834}, {295, 3.06888}, {296, 2.99979}, {297,
  2.90174}, {298, 2.77569}, {299, 2.62284}, {300, 2.44468}, {301,
  2.24292}, {302, 2.0195}, {303, 1.7766}, {304, 1.51654}, {305,
  1.24184}, {306, 0.955163}, {307, 0.65926}, {308, 0.356995}, {309,
  0.051283}, {310, -0.254924}, {311, -0.55867}, {312, -0.857023},
{313, -1.1471}, {314, -1.42611}, {315, -1.69135}, {316, -1.94027},
{317, -2.17046}, {318, -2.37969}, {319, -2.56596}, {320, -2.72745},
{321, -2.86262}, {322, -2.97016}, {323, -3.04903}, {324, -3.09846},
{325, -3.11799}, {326, -3.10742}, {327, -3.06685}, {328, -2.99668},
{329, -2.89759}, {330, -2.77052}, {331, -2.61671}, {332, -2.43764},
{333, -2.23505}, {334, -2.01088}, {335, -1.76729}, {336, -1.50665},
{337, -1.23147}, {338, -0.944395}, {339, -0.648207}, {340,
-0.345762}, {341, -0.0399792}, {342, 0.266189}, {343, 0.569788}, {344,
   0.867887}, {345, 1.15761}, {346, 1.43616}, {347, 1.70084}, {348,
  1.94911}, {349, 2.17856}, {350, 2.38698}, {351, 2.57236}, {352,
  2.73292}, {353, 2.86709}, {354, 2.97358}, {355, 3.05138}, {356,
  3.09972}, {357, 3.11813}, {358, 3.10645}, {359, 3.06479}, {360,
  2.99354}, {361, 2.89339}, {362, 2.76531}, {363, 2.61055}, {364,
  2.43058}, {365, 2.22715}, {366, 2.00222}, {367, 1.75797}, {368,
  1.49674}, {369, 1.22107}, {370, 0.933614}, {371, 0.637144}, {372,
  0.334524}, {373,
  0.0286748}, {374, -0.277451}, {375, -0.580899}, {376, -0.87874},
{377, -1.1681}, {378, -1.44618}, {379, -1.7103}, {380, -1.95792},
{381, -2.18663}, {382, -2.39424}, {383, -2.57874}, {384, -2.73834},
{385, -2.87151}, {386, -2.97697}, {387, -3.05369}, {388, -3.10093},
{389, -3.11824}, {390, -3.10545}, {391, -3.06268}, {392, -2.99035},
{393, -2.88915}, {394, -2.76007}, {395, -2.60434}, {396, -2.42348},
{397, -2.21922}, {398, -1.99354}, {399, -1.74862}, {400, -1.48682},
{401, -1.21066}, {402, -0.922822}, {403, -0.626074}, {404,
-0.323282}, {405, -0.0173701}, {406, 0.28871}, {407, 0.592002}, {408,
  0.889581}, {409, 1.17857}, {410, 1.45619}, {411, 1.71975}, {412,
  1.9667}, {413, 2.19468}, {414, 2.40147}, {415, 2.58508}, {416,
  2.74373}, {417, 2.8759}, {418, 2.98031}, {419, 3.05596}, {420,
  3.1021}, {421, 3.1183}, {422, 3.1044}, {423, 3.06053}, {424,
  2.98712}, {425, 2.88488}, {426, 2.75479}, {427, 2.59811}, {428,
  2.41635}, {429, 2.21127}, {430, 1.98484}, {431, 1.73925}, {432,
  1.47687}, {433, 1.20024}, {434, 0.912017}, {435, 0.614995}, {436,
  0.312036}, {437,
  0.0060651}, {438, -0.299964}, {439, -0.603098}, {440, -0.90041},
{441, -1.18903}, {442, -1.46617}, {443, -1.72917}, {444, -1.97546},
{445, -2.2027}, {446, -2.40866}, {447, -2.59138}, {448, -2.74909},
{449, -2.88026}, {450, -2.98362}, {451, -3.05819}, {452, -3.10323},
{453, -3.11832}, {454, -3.10331}, {455, -3.05835}, {456, -2.98386},
{457, -2.88057}, {458, -2.74948}, {459, -2.59184}, {460, -2.40919},
{461, -2.20328}, {462, -1.9761}, {463, -1.72985}, {464, -1.4669},
{465, -1.18979}, {466, -0.9012}, {467, -0.603908}, {468, -0.300786},
{469, 0.00523996}, {470, 0.311215}, {471, 0.614186}, {472,
  0.911228}, {473, 1.19947}, {474, 1.47614}, {475, 1.73856}, {476,
  1.9842}, {477, 2.21068}, {478, 2.41583}, {479, 2.59765}, {480,
  2.7544}, {481, 2.88457}, {482, 2.98689}, {483, 3.06038}, {484,
  3.10432}, {485, 3.1183}, {486, 3.10218}, {487, 3.05612}, {488,
  2.98056}, {489, 2.87622}, {490, 2.74412}, {491, 2.58554}, {492,
  2.40199}, {493, 2.19526}, {494, 1.96734}, {495, 1.72043}, {496,
  1.45692}, {497, 1.17934}, {498, 0.890372}, {499, 0.592813}, {500,
  0.289531}}
  ]
\dataplot[linecolor=red,linewidth=0.8pt,plotstyle=curve]{\mydata}

\savedata{\mydata}[
{{250, -3.04146}, {251, -3.35374}, {252, -3.63364}, {253, -3.87847},
{254, -4.08586}, {255, -4.25381}, {256, -4.3807}, {257, -4.4653},
{258, -4.5068}, {259, -4.5048}, {260, -4.45931}, {261, -4.37078},
{262, -4.24006}, {263, -4.06841}, {264, -3.85749}, {265, -3.60934},
{266, -3.32634}, {267, -3.01123}, {268, -2.66706}, {269, -2.29715},
{270, -1.90506}, {271, -1.49458}, {272, -1.06967}, {273, -0.634441},
{274, -0.193086}, {275, 0.250133}, {276, 0.690937}, {277,
  1.12507}, {278, 1.54835}, {279, 1.95668}, {280, 2.34612}, {281,
  2.71291}, {282, 3.05352}, {283, 3.36466}, {284, 3.64331}, {285,
  3.8868}, {286, 4.09277}, {287, 4.25923}, {288, 4.38458}, {289,
  4.4676}, {290, 4.5075}, {291, 4.50389}, {292, 4.45681}, {293,
  4.3667}, {294, 4.23445}, {295, 4.06132}, {296, 3.84899}, {297,
  3.5995}, {298, 3.31527}, {299, 2.99904}, {300, 2.65386}, {301,
  2.28306}, {302, 1.89022}, {303, 1.47914}, {304, 1.05378}, {305,
  0.618245}, {306,
  0.176745}, {307, -0.266461}, {308, -0.707095}, {309, -1.1409},
{310, -1.5637}, {311, -1.9714}, {312, -2.36007}, {313, -2.72596},
{314, -3.06554}, {315, -3.37553}, {316, -3.65293}, {317, -3.89507},
{318, -4.09962}, {319, -4.26459}, {320, -4.3884}, {321, -4.46984},
{322, -4.50814}, {323, -4.50293}, {324, -4.45425}, {325, -4.36257},
{326, -4.22878}, {327, -4.05417}, {328, -3.84043}, {329, -3.58962},
{330, -3.30416}, {331, -2.9868}, {332, -2.64061}, {333, -2.26894},
{334, -1.87536}, {335, -1.46368}, {336, -1.03787}, {337, -0.60204},
{338, -0.160401}, {339, 0.282786}, {340, 0.723243}, {341,
  1.15672}, {342, 1.57903}, {343, 1.9861}, {344, 2.374}, {345,
  2.73898}, {346, 3.07752}, {347, 3.38636}, {348, 3.6625}, {349,
  3.9033}, {350, 4.10642}, {351, 4.2699}, {352, 4.39216}, {353,
  4.47203}, {354, 4.50873}, {355, 4.5019}, {356, 4.45162}, {357,
  4.35838}, {358, 4.22306}, {359, 4.04697}, {360, 3.83182}, {361,
  3.57969}, {362, 3.293}, {363, 2.97452}, {364, 2.62733}, {365,
  2.25478}, {366, 1.86047}, {367, 1.4482}, {368, 1.02194}, {369,
  0.585827}, {370,
  0.144055}, {371, -0.299107}, {372, -0.739382}, {373, -1.17252},
{374, -1.59434}, {375, -2.00077}, {376, -2.38789}, {377, -2.75195},
{378, -3.08946}, {379, -3.39714}, {380, -3.67203}, {381, -3.91147},
{382, -4.11316}, {383, -4.27515}, {384, -4.39586}, {385, -4.47415},
{386, -4.50925}, {387, -4.50082}, {388, -4.44894}, {389, -4.35412},
{390, -4.21728}, {391, -4.03972}, {392, -3.82317}, {393, -3.56971},
{394, -3.2818}, {395, -2.96221}, {396, -2.61402}, {397, -2.2406},
{398, -1.84556}, {399, -1.4327}, {400, -1.00601}, {401, -0.569607},
{402, -0.127708}, {403, 0.315424}, {404, 0.755511}, {405,
  1.18831}, {406, 1.60963}, {407, 2.01542}, {408, 2.40175}, {409,
  2.7649}, {410, 3.10136}, {411, 3.40788}, {412, 3.68151}, {413,
  3.9196}, {414, 4.11985}, {415, 4.28034}, {416, 4.39951}, {417,
  4.47622}, {418, 4.50971}, {419, 4.49967}, {420, 4.4462}, {421,
  4.34982}, {422, 4.21144}, {423, 4.03241}, {424, 3.81446}, {425,
  3.55969}, {426, 3.27055}, {427, 2.94985}, {428, 2.60067}, {429,
  2.22639}, {430, 1.83062}, {431, 1.41718}, {432, 0.990058}, {433,
  0.553379}, {434,
  0.111358}, {435, -0.331737}, {436, -0.77163}, {437, -1.20407},
{438, -1.6249}, {439, -2.03003}, {440, -2.41558}, {441, -2.7778},
{442, -3.11321}, {443, -3.41857}, {444, -3.69093}, {445, -3.92767},
{446, -4.12649}, {447, -4.28548}, {448, -4.4031}, {449, -4.47822},
{450, -4.51011}, {451, -4.49847}, {452, -4.44341}, {453, -4.34545},
{454, -4.20555}, {455, -4.02505}, {456, -3.8057}, {457, -3.54962},
{458, -3.25927}, {459, -2.93746}, {460, -2.58729}, {461, -2.21215},
{462, -1.81566}, {463, -1.40164}, {464, -0.974095}, {465, -0.537144},
{466, -0.0950077}, {467, 0.348046}, {468, 0.787739}, {469,
  1.21983}, {470, 1.64014}, {471, 2.04463}, {472, 2.42937}, {473,
  2.79067}, {474, 3.12503}, {475, 3.42922}, {476, 3.70031}, {477,
  3.93569}, {478, 4.13307}, {479, 4.29056}, {480, 4.40663}, {481,
  4.48017}, {482, 4.51046}, {483, 4.49721}, {484, 4.44055}, {485,
  4.34103}, {486, 4.1996}, {487, 4.01764}, {488, 3.7969}, {489,
  3.5395}, {490, 3.24794}, {491, 2.92503}, {492, 2.57388}, {493,
  2.19789}, {494, 1.80068}, {495, 1.38609}, {496, 0.958119}, {497,
  0.520901}, {498, 0.0786556}, {499, -0.36435}, {500, -0.803838}}
  ]
\dataplot[linecolor=red,linewidth=0.8pt,plotstyle=curve]{\mydata}

\savedata{\mydata}[
{{250, -0.0983738}, {251, -2.10609}, {252, -0.681546}, {253,
-2.69393}, {254, -1.19284}, {255, -3.13342}, {256, -1.61465}, {257,
-3.40975}, {258, -1.93275}, {259, -3.51431}, {260, -2.13678}, {261,
-3.445}, {262, -2.22073}, {263, -3.20632}, {264, -2.18314}, {265,
-2.80925}, {266, -2.02713}, {267, -2.27075}, {268, -1.76036}, {269,
-1.61321}, {270, -1.39467}, {271, -0.86349}, {272, -0.945681}, {273,
-0.0519717}, {274, -0.432143}, {275, 0.788654}, {276, 0.124747}, {277,
   1.62463}, {278, 0.702177}, {279, 2.42244}, {280, 1.2766}, {281,
  3.15008}, {282, 1.82462}, {283, 3.77831}, {284, 2.32393}, {285,
  4.28177}, {286, 2.75414}, {287, 4.63997}, {288, 3.09759}, {289,
  4.83805}, {290, 3.33998}, {291, 4.86738}, {292, 3.47101}, {293,
  4.72588}, {294, 3.48466}, {295, 4.41809}, {296, 3.37953}, {297,
  3.95501}, {298, 3.15883}, {299, 3.35369}, {300, 2.83026}, {301,
  2.63652}, {302, 2.40575}, {303, 1.83041}, {304, 1.90093}, {305,
  0.965739}, {306, 1.33461}, {307, 0.07518}, {308,
  0.727976}, {309, -0.807559}, {310,
  0.103817}, {311, -1.64904}, {312, -0.514373}, {313, -2.41737},
{314, -1.1033}, {315, -3.08347}, {316, -1.64077}, {317, -3.62214},
{318, -2.10654}, {319, -4.01309}, {320, -2.48312}, {321, -4.24169},
{322, -2.75643}, {323, -4.29953}, {324, -2.91635}, {325, -4.1848},
{326, -2.95712}, {327, -3.90227}, {328, -2.87754}, {329, -3.4632},
{330, -2.68105}, {331, -2.88484}, {332, -2.37558}, {333, -2.18982},
{334, -1.97324}, {335, -1.40522}, {336, -1.48986}, {337, -0.561588},
{338, -0.944371}, {339, 0.308287}, {340, -0.358106}, {341,
  1.17061}, {342, 0.246063}, {343, 1.99188}, {344, 0.844581}, {345,
  2.7402}, {346, 1.41413}, {347, 3.38651}, {348, 1.93252}, {349,
  3.90567}, {350, 2.37956}, {351, 4.27745}, {352, 2.7378}, {353,
  4.48735}, {354, 2.99326}, {355, 4.52707}, {356, 3.13592}, {357,
  4.39492}, {358, 3.16012}, {359, 4.09584}, {360, 3.0648}, {361,
  3.64119}, {362, 2.85349}, {363, 3.04836}, {364, 2.53423}, {365,
  2.34008}, {366, 2.11924}, {367, 1.54353}, {368, 1.62441}, {369,
  0.689303}, {370, 1.06873}, {371, -0.189771}, {372,
  0.473568}, {373, -1.0599}, {374, -0.138211}, {375, -1.88762}, {376,
-0.74308}, {377, -2.6411}, {378, -1.31778}, {379, -3.29137}, {380,
-1.8402}, {381, -3.81343}, {382, -2.29027}, {383, -4.1872}, {384,
-2.65067}, {385, -4.39832}, {386, -2.90756}, {387, -4.43872}, {388,
-3.05109}, {389, -4.30687}, {390, -3.07578}, {391, -4.0079}, {392,
-2.98073}, {393, -3.55337}, {394, -2.76965}, {395, -2.96084}, {396,
-2.45073}, {397, -2.2532}, {398, -2.03633}, {399, -1.45775}, {400,
-1.54247}, {401, -0.605191}, {402, -0.988254}, {403,
  0.271566}, {404, -0.395104}, {405, 1.13869}, {406, 0.214059}, {407,
  1.96271}, {408, 0.815694}, {409, 2.71184}, {410, 1.38655}, {411,
  3.35714}, {412, 1.90458}, {413, 3.8737}, {414, 2.34975}, {415,
  4.24156}, {416, 2.70486}, {417, 4.44651}, {418, 2.95617}, {419,
  4.4806}, {420, 3.09395}, {421, 4.34247}, {422, 3.11287}, {423,
  4.0374}, {424, 3.01214}, {425, 3.57712}, {426, 2.79564}, {427,
  2.97932}, {428, 2.47167}, {429, 2.26701}, {430, 2.05272}, {431,
  1.46762}, {432, 1.5549}, {433, 0.611921}, {434,
  0.997404}, {435, -0.267125}, {436,
  0.401702}, {437, -1.13566}, {438, -0.209248}, {439, -1.96023},
{440, -0.811904}, {441, -2.70907}, {442, -1.38304}, {443, -3.35332},
{444, -1.90064}, {445, -3.86817}, {446, -2.34476}, {447, -4.23377},
{448, -2.69829}, {449, -4.43605}, {450, -2.9476}, {451, -4.46721},
{452, -3.08308}, {453, -4.32604}, {454, -3.09952}, {455, -4.01799},
{456, -2.99629}, {457, -3.55493}, {458, -2.77739}, {459, -2.9547},
{460, -2.45125}, {461, -2.24044}, {462, -2.03046}, {463, -1.43967},
{464, -1.53125}, {465, -0.583258}, {466, -0.972878}, {467,
  0.295787}, {468, -0.376879}, {469, 1.16358}, {470, 0.233761}, {471,
  1.98668}, {472, 0.83549}, {473, 2.73336}, {474, 1.4051}, {475,
  3.37482}, {476, 1.92062}, {477, 3.88634}, {478, 2.36218}, {479,
  4.24819}, {480, 2.71273}, {481, 4.44642}, {482, 2.95876}, {483,
  4.47338}, {484, 3.09077}, {485, 4.328}, {486, 3.10368}, {487,
  4.01589}, {488, 2.99698}, {489, 3.54906}, {490, 2.77477}, {491,
  2.94549}, {492, 2.44563}, {493, 2.22842}, {494, 2.02223}, {495,
  1.42549}, {496, 1.52091}, {497, 0.567629}, {498,
  0.960976}, {499, -0.312111}, {500, 0.364028}}
]
\dataplot[linecolor=black,linewidth=0.8pt,plotstyle=dots]{\mydata}
\psaxes[linecolor=gray,Ox=250,Oy=0,Dx=50,dx=50,Dy=2,dy=2]{->}(250,0)(230,-5)(520,5)

  \rput(520,-1){$N$}
  \rput(528,4.5){$\alpha' \sin(\beta'+NV/2)$}
\rput(537,3){$\alpha'' \sin(\beta''+NV/2)$}

\end{pspicture}
\caption{$C_{121}(N)*N^2 e^{-NU/2}$ for $250 \leqslant N \leqslant 500$ \label{qfig}}
\end{center}
\end{figure}


Since $|w_0|<1$ and $U>0$, Conjectures \ref{coonj1} and \ref{coonj2} certainly imply that $C_{011}(N)$ and $C_{121}(N)$  become arbitrarily large  (with $C_{011}(N)$ on the order of the square of $C_{121}(N)$) and do not converge to any limit.
In \cite{OS} the proof of a slightly weaker version of Conjecture \ref{coonj1}, enough to disprove Rademacher's conjecture, is being finished. The  chief difficulty  comes from estimating the error terms and showing they are indeed less than the expected main terms.

Though Rademacher's original conjecture appears to be false, we can speculate that another version of it may be true, perhaps modifying the sequence by averaging the coefficients or introducing
 factors to ensure convergence. It is remarkable that there  seems to be a connection
 between $C_{hk\ell}(N)$ and  $C_{hk\ell}(\infty)$  when $N$ is {\em small}  as noted in \cite[Sect. 4]{SZ} and by Rademacher himself \cite[p. 302]{Ra}. We include more evidence of this phenomenon in Table \ref{small} by setting $C_{hk\ell}(\star):=\frac{1}{100}\sum_{N=1}^{100}C_{hk\ell}(N)$.
\begin{table}[h]
\begin{center}
\begin{tabular}{c|c|c||c|c|c}
$h$ & $k$ & $\ell$ & $C_{hk\ell}(\star)$ & $C_{hk\ell}(\infty)$ & $\left| 1-C_{hk\ell}(\star)/C_{hk\ell}(\infty) \right|$ \\ \hline
$0$ & $1$ & $1$ & $-0.2812 $ & $-0.2929$ & $0.04005$\\
$1$ & $2$ & $1$ & $0.09511$ & $0.09388$ & $0.01309$\\
$1$ & $3$ & $1$ & $0.02429-0.02899 i$ & $0.02417-0.02881 i$ & $0.005911$\\
$1$ & $4$ & $1$ & $0.007312-0.01775 i$ & $0.007252-0.01751 i$ & $0.01332$\\
$0$ & $1$ & $2$ & $0.1921$ & $0.1898$ & $0.01219$\\
$1$ & $2$ & $2$ & $0.01510$ & $0.01531$ & $0.01392$\\
$1$ & $3$ & $2$ & $-0.0009181-0.002514 i$ & $-0.0009364-0.002573 i$ & $0.02233$\\
$1$ & $4$ & $2$ & $-0.0006919-0.0002846 i$ & $-0.0007183-0.0002975 i$ & $0.03771$
\end{tabular}
\caption{} \label{small}
\end{center}
\end{table}

\section{A conjecture of Sills and Zeilberger} \label{zel}
In \cite[Sect. 3]{SZ}, Sills and Zeilberger define\footnote{changing their notation slightly from $P_{01(N-r)}(N)$ to $P_{01r}(N)$}
\begin{equation*}
P_{01r}(N) :=(-1)^N N!  \cdot (-4)^r r! \cdot C_{01(N-r)}(N).
\end{equation*}
For each $r$, by solving a recursion,  they   prove `top down' formulas such as
\begin{equation*}
    P_{010}(N)=1, \quad P_{011}(N)=N^2-N, \quad P_{012}(N)=N^4-\frac{22}{9}N^3+\frac{13}{3}N^2-\frac{26}{9}N,
\end{equation*}
with a procedure they automated.

\begin{conj}[Sills and Zeilberger \cite{SZ}] \label{szconj}
For each $r \gqs 1$, $P_{01r}(N)$ is a monic, alternating, convex polynomial in $N$ of degree $2r$ whose only real roots are $0$ and $1$.
\end{conj}

By extending our work from Section \ref{tw}, we find a formula for $P_{01r}(N)$ and prove part of Conjecture \ref{szconj}. To begin, rearrange \eqref{c100} into
\begin{equation*} \label{c200}
    \frac{P_{01r}(N)}{(-4)^r r!} = \left[ \text{coeff. of }z^{r} \right]e^z \left(\frac{z}{e^z - 1}\right)^{r+1} \left(\frac{z}{e^z - 1}\right)^{-N}
\left(\frac{z}{e^z - 1}\right)\left(\frac{2z
}{e^{2z} - 1}\right) \cdots \left(\frac{N z}{e^{Nz} - 1}\right).
\end{equation*}
As in \eqref{ezz} this becomes
\begin{equation} \label{c300}
    \frac{P_{01r}(N)}{(-4)^r r!} = \left[ \text{coeff. of }z^{r} \right]e^z \left(\frac{z}{e^z - 1}\right)^{r+1} \exp\left( \sum_{j=1}^r \frac{(-1)^{j+1} B_j z^j}{j \cdot j!} \Bigl(s_j(N)-N \Bigr)\right).
\end{equation}

The expressions $s_1(N)-N$, \dots, $s_r(N)-N$ contain the only appearance of $N$ on the right side of  \eqref{c300}.  Recalling the definition in \eqref{smn}, we may write
\begin{equation} \label{poly1}
    s_m(N)=\frac{1}{m+1}\sum_{j=0}^m \binom{m+1}{j+1} (-1)^{m-j} B_{m-j} \cdot N^{j+1} \qquad (m,N \in \Z_{\gqs 0})
\end{equation}
by rearranging \cite[Eq. (2.3)]{Ra} or \cite[Eq. (6.78)]{Knu} for example. So  $s_m(N)-N$ is a polynomial of degree $m+1$ in $N$ and we may  replace $N$ by an arbitrary variable $x$.

At this point, we need to introduce the Stirling numbers $\stira{n}{m}$. They denote the number of permutations of a set of size $n$ that have $m$ disjoint  cycles. For an in-depth discussion of both types of Stirling numbers $\stirb{n}{m},\stira{n}{m}$,  including their
history and notation, see \cite[Chap. 6]{Knu}, \cite{K2}. They satisfy
\begin{align}\label{sb1}
    \sum_{m=0}^n \stirb{n}{m} x(x-1) \cdots (x-m+1) & = x^n,\\
    x(x+1) \cdots (x+n-1) & = \sum_{m=0}^n \stira{n}{m} x^m. \label{sa1}
\end{align}
Also, the analog of \eqref{norm} is
\begin{equation}
    B_n^{(r)}  = (-1)^n \frac{ r}{r-n}\stira{r}{r-n}\Big/\binom{r}{r-n}, \qquad (n \in \Z_{\gqs 0}, \ r \in \Z_{\gqs n+1}) \label{norp}
\end{equation}
from \cite[Eq. (2)]{Ca}.

\begin{lemma}\label{lx}
For $r \gqs 0$,
\begin{equation} \label{lo}
   e^z \left( \frac{z}{e^z -1}\right)^{r+1}   = \sum_{m =0}^r (-1)^m \stira{r}{r-m} \Big/ \binom{r}{r-m} \frac{z^m}{m!} + O(z^{r+1}).
\end{equation}
\end{lemma}
\begin{proof}
The case $r=0$ is easy so we assume $r\gqs 1$.
The coefficient of $z^m/m!$ on the left of \eqref{lo} is $B_m^{(r+1)}(1)$.
By \eqref{how} and \eqref{norp} we have
\begin{alignat}{2}\label{qt1}
    B_m^{(r+1)}(1) & = \frac{r-m}{r} B_m^{(r)} \qquad & & (0 \lqs m)\\
    & = \frac{r-m}{r} (-1)^m \frac{ r}{r-m}\stira{r}{r-m}\Big/\binom{r}{r-m} \qquad & & (0 \lqs m \lqs r-1) \notag
\end{alignat}
so that
\begin{equation}\label{qt2}
    B_m^{(r+1)}(1) = (-1)^m \stira{r}{r-m}\Big/\binom{r}{r-m} \qquad (0 \lqs m \lqs r-1).
\end{equation}
We may verify that \eqref{qt2} is also true for $m=r$ since $B_m^{(m+1)}(1)=0$ by \eqref{qt1}.
The lemma follows.
\end{proof}

\begin{theorem} \label{poll}
For $r \gqs 1$, $P_{01r}(x)$ is a monic polynomial in $x$ of degree $2r$ with $0$ and $1$ as roots. It is given by
\begin{multline} \label{wer}
    P_{01r}(x) = 4^r \sum_{m=0}^r (-1)^m \stira{r}{m} m! \sum_{1j_1+2j_2+ \cdots + m j_{m} = m}
     \frac{1}{j_1! j_2! \cdots j_m!} \\
    \times
     \left(\frac{B_1}{1 \cdot 1!}\bigl(s_1(x)-x\bigr) \right)^{j_1}  \cdots \left(\frac{(-1)^{m-1}B_{m}}{m \cdot m!}\bigl(s_{m}(x)-x \bigr) \right)^{j_{m}}.
\end{multline}
\end{theorem}
\begin{proof}
The formula \eqref{wer} follows from Lemma \ref{lx} and \eqref{c300}, showing that $P_{01r}(x)$ is a  polynomial in $x$.
The term in \eqref{wer} corresponding to $(j_1,j_2, \dots ,j_m)$ has degree
$$
2j_1+3j_2+ \dots +(m+1)j_m = j_1+j_2+ \dots +j_m+m
$$
with $m \lqs r$ and $j_1 \lqs m$, $j_2 \lqs m/2$, \dots, $j_m \lqs 1$. Hence, the maximum degree term has $m=j_1=r$ and $j_2=j_3= \dots =j_r=0$. This $2r$-degree term contributes
\begin{equation*}
    4^r (-1)^r r! \frac{1}{r!} \left(\frac{B_1}{1 \cdot 1!}\bigl(s_1(x)-x\bigr) \right)^{r} = 4^r  \left(\frac{x^2-x}{4} \right)^{r}=x^{2r}+O(x^{2r-1})
\end{equation*}
to \eqref{wer}, proving that $P_{01r}(x)$ is  monic of degree $2r$.

Since $s_m(1)=1^m$ by definition and $s_m(n)$ in \eqref{poly1} has no constant term, we see that $0$ and $1$ are roots of $s_m(x)-x$ for $m \gqs 1$. Therefore $0$ and $1$ are roots of all terms on the right side of \eqref{wer} except possibly the term with $j_1=j_2= \dots =j_m=m=0$. However this term must be zero because $\stira{r}{0}=0$ for $r\gqs 1$. Therefore $0$ and $1$ are roots of $P_{01r}(x)$, completing the proof.
\end{proof}

With Theorem \ref{poll}, we have proven part of Sills and Zeilberger's Conjecture \ref{szconj} but it remains to show that $P_{01r}(x)$ is convex, i.e. $P_{01r}''(x) \gqs 0$, and  that the coefficients of $P_{01r}(x)$  alternate in sign. Note that convexity implies that there are no more real roots after $0$ and $1$.

As we showed in Theorem \ref{poll}, the coefficient of $x^{2r}$ is $1$. The same methods allow us to calculate the next highest terms:
\begin{equation} \label{alte}
    P_{01r}(x) = x^{2r}-\frac{2r^2+7r}{9} x^{2r-1} + \frac{4r^4+12 r^3+287 r^2-303r}{162} x^{2r-2} + \dots
\end{equation}
for example, proving the formulas observed in \cite[Remark 3.4]{SZ}. It follows from \eqref{alte} that the coefficients of the three highest degree terms of $P_{01r}(x)$ alternate. 
We may also examine the lowest degree terms.

\begin{theorem} \label{poi}
For $r\gqs 1$, the coefficient of $x$ in the polynomial $P_{01r}(x)$ is always negative with
\begin{equation} \label{x}
   \left[ \text{\rm coeff. of }x \right] P_{01r}(x) =  4^r \sum_{i=1}^r \stira{r}{i} \frac{B_i}{i}   \Bigl(1-(-1)^i B_i \Bigr).
\end{equation}
\end{theorem}
\begin{proof}
We prove \eqref{x} first. Since $x$ divides $s_m(x)-x$, the only terms in \eqref{wer} that can contribute to the $x$ term have all of $j_1=j_2= \dots =j_m=0$ except for one $j_i=1$. The coefficient of $x$ in $s_i(x)-x$ is $(-1)^i B_i -1$ so
\begin{equation*}
    \left[ \text{\rm coeff. of }x \right] \frac{(-1)^{i-1}B_{i}}{i \cdot i!}\bigl(s_{i}(x)-x \bigr) = \frac{B_{i}}{i \cdot i!}\bigl((-1)^i-B_i \bigr).
\end{equation*}
Hence
\begin{equation*}
    \left[ \text{\rm coeff. of }x \right] P_{01r}(x)  = 4^r \sum_{i=1}^r (-1)^i \stira{r}{i}  i! \frac{B_{i}}{i \cdot i!}\bigl((-1)^i-B_i \bigr)
\end{equation*}
and we have verified \eqref{x}.

To complete the proof we need to check these coefficients are negative. We recall a few facts about Bernoulli numbers, see \cite[Chap. 1]{Ra} for example. We have $B_i=0$ for $i \gqs 3$ and odd. Also
\begin{equation*}
    B_i=(-1)^{i/2+1} 2 \zeta(i) \frac{i!}{(2\pi)^i}  \qquad (i\gqs 2, \text{ even})
\end{equation*}
where $\zeta$ is the Riemann zeta function and $1<\zeta(i) \lqs \zeta(2)=\pi^2/6<2$ for $i\gqs 2$. Hence
\begin{equation}\label{bn1}
  2\frac{i!}{(2\pi)^i} <  \left|B_i \right| < 4\frac{i!}{(2\pi)^i}  \qquad (i\gqs 2, \text{ even}).
\end{equation}
Also, by Stirling's formula \cite[pp. 26-28]{Ra},
\begin{equation}\label{bn2}
    2\sqrt{i} \left(\frac ie\right)^i < i! < 3\sqrt{i} \left(\frac ie\right)^i \qquad (i \gqs 1).
\end{equation}

It follows that $|B_i|>1$ exactly for $i \gqs 14$ and even ($B_{14}=7/6$) and so
\begin{equation*}
    \stira{r}{i} \frac{B_i}{i}   \Bigl(1-(-1)^i B_i \Bigr)  < 0 \qquad (i\gqs 14, \text{ even}).
\end{equation*}
By omitting all the terms with $14 \lqs i <r$ (and assuming $r \gqs 14$ and even for simplicity) we find
\begin{align} \notag
   4^{-r} \left[ \text{\rm coeff. of }x \right] P_{01r}(x) & < \stira{r}{r} \frac{B_r}{r}   \Bigl(1- B_r \Bigr) + \sum_{i=1}^{12} \stira{r}{i} \frac{B_i}{i}   \Bigl(1-(-1)^i B_i \Bigr) \\
   & <  \stira{r}{r} \frac{B_r}{r}   \Bigl(1- B_r \Bigr) + 2\sum_{i=1}^r \stira{r}{i}. \label{lkj}
\end{align}
Setting $x=1$ in \eqref{sa1} shows that $\sum_{i=1}^r \stira{r}{i} =r!$. Using the bounds \eqref{bn1}, \eqref{bn2} in \eqref{lkj} then yields
\begin{equation*}
    4^{-r} \left[ \text{\rm coeff. of }x \right] P_{01r}(x) < 8 r!\left( 1-\frac{1}{\sqrt{r}}\left(\frac{r}{4\pi^2 e} \right)^r \right)
\end{equation*}
showing that the coefficients are certainly negative for $r \gqs 110$. Checking directly the coefficients for $1 \lqs r < 110$ completes the proof.
\end{proof}

The same method allows us to prove formulas for the coefficients of $x^2$, $x^3$ etc. For example the next result has a similar proof to Theorem \ref{poi}.

\begin{theorem}
For $r\gqs 1$, the coefficient of $x^2$ in the polynomial $P_{01r}(x)$ is always positive and given by
\begin{multline*}
    4^r \left( \frac 14 \stira{r}{1} -\frac 1{24} \stira{r}{2}+\frac 12 \sum_{i=1}^{r/2} \stira{r}{2i}\binom{2i}{i} \left(\frac{B_i}{i}\Bigl(1-(-1)^i B_i \Bigr) \right)^2\right. \\
    \left. +\sum_{i=1}^{r} \sum_{j=i+1}^{r} \stira{r}{i+j}\binom{i+j}{i} \frac{B_i}{i}\Bigl(1-(-1)^i B_i\Bigr) \frac{B_j}{j}\Bigl(1-(-1)^j B_j\Bigr) \right).
\end{multline*}
\end{theorem}

So far, our techniques  allow us to check one-by-one that the coefficients of $P_{01r}(x)$ are alternating.
We next show that a large piece of $P_{01r}(x)$ is  alternating, though for very large $r$ it is not convex.
Consider the sum of terms in \eqref{wer} with $j_1=m$ and $j_2=j_3=\dots =j_m=0$. Call this subsum $M_{01r}(x)$ so that
\begin{align*}
    M_{01r}(x) & = 4^r \sum_{m=0}^r (-1)^m \stira{r}{m} m!
     \frac{1}{m!}
     \left(\frac{B_1}{1 \cdot 1!}\bigl(s_1(x)-x\bigr) \right)^{m}\\
     & = 4^r \sum_{m=0}^r  \stira{r}{m} \left(\frac{x(x-1)}4 \right)^m\\
     & = \prod_{j=0}^{r-1}(x(x-1)+4j).
\end{align*}

\begin{lemma}
For $r\gqs 1$, $M_{01r}(x)$ is alternating but not always convex.
\end{lemma}
\begin{proof}
It is easy to see that $M_{01r}(x)$ is alternating: just note that $M_{01r}(-x)=\prod_{j=0}^{r-1}(x^2+x+4j)$ has all its coefficients positive. To check convexity, a straightforward calculation shows
\begin{equation*}
    M_{01r}''(1/2)=2  \left( 1-\sum_{a=1}^{r-1} \frac1{16 a-1}\right) \prod_{j=1}^{r-1}(4j-1/4).
\end{equation*}
Since $\sum_{a=1}^{r-1}1/a >\log r$ it follows that $M_{01r}(x)$ is not convex for $r>e^{16}$.
\end{proof}

This perhaps casts doubt on $P_{01r}(x)$ being convex for all $r$. However, if we could show that $M_{01r}(x)$ is the dominant part of $P_{01r}(x)$, having larger coefficients  than the remaining piece $P_{01r}(x)-M_{01r}(x)$, this would imply that $P_{01r}(x)$ is also alternating. Hopefully these issues can be pursued  in a future work.

We finally remark that throughout this section we have assumed $h/k=0/1$, so it would be interesting to see what happens in the general case. What are the expressions for $C_{hk(N-r)}(N)$ analogous to $P_{01r}(N)$?

\section{Further formulas} \label{fur}

 Andrews supplied the first formula for $C_{011}(N)$, expressing it in \cite[Theorem 1]{An} as
 \begin{multline} \label{Andrews}
    C_{011}(N)=\frac{-1}{N!}\sum_{j_2=1}^1 \sum_{j_3=1}^2 \cdots \sum_{j_N=1}^{N-1} e\left(-\frac{j_2}{2} -\frac{j_3}{3}- \cdots -\frac{j_N}{N}\right)
    \\
    \times  \sum_{r_2+\cdots + r_N=N-1} \left( \frac{e(j_2/2)}{1-e(j_2/2)}\right)^{r_2} \cdots \left( \frac{e(j_N/N)}{1-e(j_N/N)}\right)^{r_N}
 \end{multline}
 for $e(z):=e^{2\pi i z}$.
  The starting point for the general case of Andrews' formula is the  identity
\begin{equation}\label{and}
    C_{hk\ell}(N) = \frac{1}{(N-\ell)!}\frac{d^{N-\ell}}{dx^{N-\ell}} \left. \frac{(x-\rho)^{N}}{(1-x)(1-x^2) \cdots (1-x^N)}\right|_{x=\rho}
\end{equation}
for $\rho=e^{2\pi i h/k}$ as before. With a different treatment of \eqref{and} we obtain
\begin{equation} \label{ano}
    C_{01\ell}(N)=
\frac{(-1)^{N}}{N!}
\sum_{m=1}^{N-\ell} (-1)^m
\sum_{ \substack{i_1, i_2, \dots ,i_m \geqslant 1 \\
i_1+i_2+ \cdots +i_m=N-\ell \\
1\leqslant r_1 \leqslant r_2 \leqslant \cdots \leqslant r_m \leqslant N }}
 \frac{
\binom{r_1}{i_1+1} \binom{r_2}{i_2+1} \cdots \binom{r_m}{i_m+1}}
{r_1 \cdot r_2 \cdots r_m}
\end{equation}
as a special case of Theorem \ref{big} below.
We begin with some straightforward lemmas.

\begin{lemma}[Fermat] \label{mabel}
We have
\begin{equation}\label{bino}
\binom{j}{j}+\binom{j+1}{j}+\binom{j+2}{j}+ \cdots +\binom{r-1}{j}= \binom{r}{j+1}.
\end{equation}
\end{lemma}
\begin{proof}
Write the left side of \eqref{bino} as
\begin{equation}\label{bino2}
    \binom{j+1}{0}+\binom{j+1}{1}+\binom{j+2}{2}+ \cdots +\binom{r-1}{r-1-j}.
\end{equation}
By successively combining the first two  terms of \eqref{bino2} using Pascal's identity, the sum collapses and
we obtain the lemma.
\end{proof}

It follows from Lemma \ref{mabel} that
\begin{align}
    \left. \frac{d^j}{dx^j}  \Bigl( 1+x+x^2+ \cdots +x^{r-1}\Bigr) \right|_{x=1} & = j!+\frac{(j+1)!}{1!}+ \cdots + \frac{(r-1)!}{(r-1-j)!} \notag\\
    & = j!\left( \binom{j}{j}+\binom{j+1}{j}+ \cdots +\binom{r-1}{j}\right) \notag\\
    & = j! \binom{r}{j+1}. \label{ee}
\end{align}

Now define
\begin{align}\label{hr}
    h_\rho(r,x) & :=\begin{cases} 1+x+x^2+\cdots +x^{r-1} & \text{ \ if \ } k \mid r\\
    1-\rho^r x^r & \text{ \ if \ } k \nmid r
    \end{cases}\\
    G_\rho(r,i) & := \frac{-1}{h_\rho(r,1) \cdot i!} \left. \frac{d^i}{dx^i} h_\rho(r,x) \right|_{x=1} \label{ghr}
\end{align}

\begin{lemma} \label{mabel2}
For $r$, $i \gqs 1 $ we have
\begin{equation}\label{gro}
G_\rho(r,i) =  \frac{-1}{r}\binom{r}{i+1}  \text{ \ \ if \ \ } k \mid r, \qquad
   G_\rho(r,i) = \frac{\rho^r}{1-\rho^r}\binom{r}{i}  \text{ \ \ if \ \ } k \nmid r.
\end{equation}
\end{lemma}
\begin{proof}
The left side of \eqref{gro} follows from \eqref{ee}. The right side is a simpler calculation, with
$$
\left. \frac{d^i}{dx^i}  \Bigl( 1-\rho^r x^r \Bigr) \right|_{x=1} = -\rho^r i! \binom{r}{i} \qquad (i \gqs 1).
$$
\end{proof}

\begin{theorem} \label{big}
With $\rho =e^{2\pi i h/k}$, define $G_\rho(r,i)$ as in \eqref{ghr} or \eqref{gro} and set $s:=\lfloor N/k\rfloor$. Then
\begin{equation}\label{3}
C_{hk\ell}(N)=
\frac{(-1)^{s}\rho^\ell}{ k^{2s} s!}\left( \prod_{d=1}^{N-k s} \frac{1}{1-\rho^d}\right)
\sum_{m=1}^{s-\ell}
\sum_{ \substack{i_1, i_2, \dots ,i_m \geqslant 1 \\
i_1+i_2+ \cdots +i_m=s-\ell \\
1\leqslant r_1 \leqslant r_2 \leqslant \cdots \leqslant r_m \leqslant N }}
 G_\rho(r_1,i_1) G_\rho(r_2,i_2) \cdots G_\rho(r_m,i_m).
\end{equation}
\end{theorem}
\begin{proof}
Changing the variable from $x$ to $\rho x$ in \eqref{and} produces
\begin{align*}
   (N-\ell)! C_{hk\ell}(N) & = \rho^\ell \frac{d^{N-\ell}}{dx^{N-\ell}} (x-1)^{N} \prod_{1 \lqs d \lqs N}  \left. \frac{1}{1-\rho^d x^d}\right|_{x=1}\\
   & = \rho^\ell (-1)^s \frac{d^{N-\ell}}{dx^{N-\ell}}  (x-1)^{N-s} \prod_{\substack{1 \lqs d \lqs N \\ k \mid d}}  \frac{1}{1+x+\cdots +x^{d-1}} \prod_{\substack{1 \lqs d \lqs N \\ k \nmid d}} \left. \frac{1}{1-\rho^d x^d}\right|_{x=1}\\
   & = \rho^\ell (-1)^s \sum_{a+b=N-\ell} \binom{N-\ell}{a}\frac{d^{a}}{dx^{a}} \left. (x-1)^{N-s} \right|_{x=1} \cdot \frac{d^{b}}{dx^{b}} \prod _{1 \lqs d \lqs N}  \left. \frac{1}{h_\rho(d,x)}\right|_{x=1}.
\end{align*}
The only non-zero terms have $a=N-s$ and hence $b=s-\ell$. Therefore
\begin{align}
    C_{hk\ell}(N) & = \frac{\rho^\ell (-1)^s}{(s-\ell)!} \frac{d^{s-\ell}}{dx^{s-\ell}} \prod _{1 \lqs d \lqs N}  \left. \frac{1}{h_\rho(d,x)}\right|_{x=1} \notag\\
    & = \frac{\rho^\ell (-1)^s}{(s-\ell)!}
    \sum_{m_1+ \cdots + m_N = s-\ell} \binom{s-\ell}{m_1, \dots ,m_N} \prod_{r=1}^N
    \frac{d^{m_r}}{dx^{m_r}}   \left. \frac{1}{h_\rho(r,x)}\right|_{x=1} \label{toller}
\end{align}
using Leibnitz' formula. Fa\`a di Bruno's formula (the symmetric form, due to Ces\`aro and Riordan - see \cite[Eq. (2.2)]{Jo}) tells us that
$$
\frac{d^m}{dx^m}g(f(x))=\sum_{d=0}^m g^{(d)}(f(x)) \frac{1}{d!} \sum_{ \substack{i_1,i_2, \dots, i_d \geqslant 1 \\
 i_1+i_2+ \cdots +i_d=m}} \binom{m}{i_1,i_2, \dots ,i_d} f^{(i_1)}(x) \cdots f^{(i_d)}(x)
$$
where the inner sum is defined as $\delta_{m,0}$ when  $d=0$.
Apply this with $g(x)=1/x$ and $f(x)=h_\rho(r,x)$ to get
$$
\frac{d^m}{dx^m} \frac{1}{h_\rho(r,x)}=\sum_{d=0}^m  \frac{(-1)^d}{h_\rho(r,x)^{d+1}} \sum_{ \substack{i_1,i_2, \dots, i_d \geqslant 1 \\
 i_1+i_2+ \cdots +i_d=m }} \binom{m}{i_1,i_2, \dots ,i_d} h^{(i_1)}_\rho(r,x) \cdots h^{(i_d)}_\rho(r,x).
$$
Therefore
\begin{equation}\label{lotr}
\frac{d^m}{dx^m} \left. \frac{1}{h_\rho(r,x)}\right|_{x=1} =\frac{m!}{h_\rho(r,1)} \sum_{d=0}^m  \sum_{ \substack{i_1,i_2, \dots, i_d \geqslant 1 \\
 i_1+i_2+ \cdots +i_d=m}}  G_\rho(r,i_1) \cdots  G_\rho(r,i_d)
\end{equation}
 and inserting \eqref{lotr} into \eqref{toller},
\begin{align*}
    C_{hk\ell}(N) & = \rho^\ell (-1)^s \left( \prod_{w=1}^N \frac{1}{h_\rho(w,1)}\right) \sum_{m_1+ \cdots + m_N = s-\ell} \prod_{r=1}^N \left(
    \sum_{d=0}^{m_r}  \sum_{ \substack{i_1,i_2, \dots, i_d \geqslant 1 \\
 i_1+i_2+ \cdots +i_d=m_r}}  G_\rho(r,i_1) \cdots  G_\rho(r,i_d)\right)\\
 & = \rho^\ell (-1)^s \left( \prod_{w=1}^N \frac{1}{h_\rho(w,1)}\right) \sum_{m=1}^{s-\ell} \sum_{ \substack{i_1, i_2, \dots ,i_m \geqslant 1 \\
i_1+i_2+ \cdots +i_m= s-\ell \\
1\leqslant r_1 \leqslant r_2 \leqslant \cdots \leqslant r_m \leqslant N }}
 G_\rho(r_1,i_1) G_\rho(r_2,i_2) \cdots G_\rho(r_m,i_m).
\end{align*}
Finally,
\begin{align*}
\prod_{w=1}^N h_\rho(w,1) & = \prod_{\substack{ 1 \lqs w \lqs N \\ k \mid w}} h_\rho(w,1) \cdot \prod_{\substack{ 1 \lqs w \lqs k s \\ k \nmid w}} h_\rho(w,1) \cdot \prod_{\substack{ k s +1 \lqs w \lqs N \\ k \nmid w}} h_\rho(w,1)\\
& =\prod_{\substack{ 1 \lqs w \lqs N \\ k \mid w}} w \cdot \prod_{\substack{ 1 \lqs w \lqs k s \\ k \nmid w}} (1-\rho^w) \cdot \prod_{\substack{ k s +1 \lqs w \lqs N \\ k \nmid w}} (1-\rho^w)\\
& =k^s s! \biggl(\prod_{ 1 \lqs w \lqs k-1} (1-\rho^w) \biggr)^s \prod_{ 1 \lqs w \lqs N-k s} (1-\rho^w).
\end{align*}
With Lemma \ref{rou}, this completes the proof. \end{proof}

\begin{cor} \label{zco}
Recursively define
\begin{align*}
    Q_\rho(0,a) & := 0, \qquad (a \gqs 0), \\
    Q_\rho(N,0) & := 1, \qquad (N \gqs 1), \label{sz1}
\end{align*}
and, with $G_\rho(N,b)$ as in \eqref{gro},
\begin{equation} \label{sz2}
    Q_\rho(N,a) := Q_\rho(N-1,a)+\sum_{b=1}^a Q_\rho(N,a-b) G_\rho(N,b), \qquad (N \gqs 1, \ a \gqs 1).
\end{equation}
Then, for $s:=\lfloor N/k\rfloor$,
\begin{equation} \label{sz3}
    C_{hk\ell}(N)=
\frac{(-1)^{s}\rho^\ell}{ k^{2 s} s!}\left(\prod_{d=1}^{N-k s} \frac{1}{1-\rho^d}\right) Q_\rho(N,s -\ell).
\end{equation}
\end{cor}
\begin{proof}
The above definition of $Q_\rho(N,a)$ corresponds to
\begin{equation}\label{label}
    Q_\rho(N,a) =\sum_{m=1}^{a} \sum_{ \substack{i_1, i_2, \dots ,i_m \geqslant 1 \\
i_1+i_2+ \cdots +i_m=a \\
1\leqslant r_1 \leqslant r_2 \leqslant \cdots \leqslant r_m \leqslant N }}
 G_\rho(r_1,i_1) G_\rho(r_2,i_2) \cdots G_\rho(r_m,i_m).
\end{equation}
Since $r_m$ in the  summation \eqref{label} is either  $\lqs N-1$ or equal to $N$, we obtain \eqref{sz2}.
Then Theorem \ref{big} implies \eqref{sz3}.
\end{proof}

Corollary \ref{zco} describes, essentially, the recursion used by Sills and Zeilberger for their computations. See \cite[Sect. 2.2]{SZ} for their derivation in the $h/k=0/1$ case.

\bibliography{raddata}

\end{document}